 \documentclass[preprint,12pt]{elsarticle}

\usepackage{amssymb}
\usepackage[utf8]{inputenc}
\usepackage{comment}
\usepackage{multirow}
\usepackage{subfig}
\usepackage{float}
\usepackage{amsmath}
\usepackage{amsthm}
\usepackage{listings}
\usepackage{amsfonts}
\usepackage{pgf,tikz,pgfplots}
\pgfplotsset{compat=1.15}
\usepackage{mathrsfs}
\usepackage{comment}
\usetikzlibrary{arrows}
\usepackage[normalem]{ulem}
\usepackage{appendix}
\usepackage{bm}
\usepackage{color}

\usepackage{hyperref}
\hypersetup{
  colorlinks   = true, 
  urlcolor     = blue, 
  linkcolor    = blue, 
  citecolor   = purple 
}

\usepackage[colorinlistoftodos,textwidth=1.5cm]{todonotes}

\DeclareGraphicsExtensions{.pdf,.eps,.ps,.jpg,.epsi,.png}

\graphicspath{{images/}{./}}

\newtheorem{Definition}{Definition}
\newtheorem{Algorithm}{Algorithm}
\newtheorem{Theorem}{Theorem}

\setcitestyle{square}

\journal{Applied Numerical Mathematics}

\begin{document}

\begin{frontmatter}


 \title{Implicit and semi-implicit well-balanced finite-volume methods for systems of balance laws}
\author{I. G\'omez-Bueno\corref{cor1}\fnref{label2}}
\ead{igomezbueno@uma.es}
\cortext[cor1]{Corresponding author.}
\address[label2]{Departamento de An\'alisis Matem\'atico, Estad\'istica e I.O. y Matem\'atica Aplicada. Facultad de Ciencias, Campus de Teatinos. Universidad de M\'alaga, 29071 M\'alaga, Spain}

\author{S. Boscarino \fnref{label3}}

\address[label3]{Dipartimento di Matematica ed Informatica. Viale Andrea Doria, 6, 95125. University of Catania. Catania, Italy. }

\author{M.J. Castro\fnref{label2}}

\author{C. Par\'es\fnref{label2}}

\author{G. Russo\fnref{label3}}

\begin{abstract}
The aim of this work is to design implicit and semi-implicit high-order well-balanced finite-volume numerical methods for 1D systems of balance laws.
The strategy introduced by two of the authors in \cite{CastroPares2019} for explicit schemes based on the application of a well-balanced reconstruction operator has been applied.
The well-balanced property is preserved when quadrature formulas are used
to approximate the averages and the integral of the source term in the cells. Concerning the time
evolution, this technique is combined with a time discretization method of type RK-IMEX or RK-implicit (see \cite{pareschi2005implicit,boscarino2009class}). The methodology will
be applied to several systems of balance laws.

\end{abstract}


\end{frontmatter}


\section{Introduction}

Numerous physical systems are described by evolutionary partial differential equations which have the structure of hyperbolic systems of balance laws of the form

\begin{equation}\label{PDE_generalproblem}
    u_t+f(u)_x=S(u)H_x,
\end{equation}
where 
\begin{itemize}
    \item $u(x,t)$ takes values on an open convex set $\Omega \subset \mathbb{R}^N$;
    \item $f:\Omega \longrightarrow \mathbb{R}^N$ is the physical flux function;
    \item the source term is written in the form $S(u)H_x$, where $S:\Omega \longrightarrow \mathbb{R}^N$ and $H:\mathbb{R} \longrightarrow \mathbb{R}$ is a known function (possibly the identity function $H(x)=x$).
\end{itemize}

The system \eqref{PDE_generalproblem} has nontrivial stationary solutions  that satisfy the ODE system:
 \begin{equation}\label{ODE_stationarysolutions}
  f(u)_x = S(u)H_x,   
 \end{equation}
or 
$$
J(u)u_x = S(u)H_x,
$$
where $J(u)$ is the Jacobian of the flux function. 
A numerical method is said to be well-balanced if it preserves (in some sense) all or a representative set of steady solutions of \eqref{PDE_generalproblem}. The development of numerical schemes satisfying the well-balanced property is a big issue in the simulation of small perturbations of stationary solutions in many geophysical problems, such as the tsunami waves in the Ocean. Many authors have already deal with the design of well-balanced methods: see, for example, \cite{Audusse,Bouchut,bicaper,EFN-03,EFN-CRAS,Chandrashekar15,Chandrashekar17, Desveaux16,desveauxRipa2016,Gosse00,Gosse01,Gosse02,Kappeli14,LR96,LR97,LeVeque98, Lukacova,Noelle06,Noelle07,Pelanti,PerSi1,PerSi2,Russo,Tang04,touma2016,Xing06, Klingenberg19,Klingenberg20,Kurganov2016,Kurganov2018,Chertok2015,handbook,
 franck2016finite,grosheintz2020well,kappeli2016wellb} and their references. In earlier papers, some of the authors introduced a general procedure to build explicit high-order well-balanced numerical methods whose key was the design of high-order well-balanced reconstruction operators (see \cite{CastroPares2019,GomezCastroPares2020,gomez2021well,gomez2021collocation}). The aim of this work is to extend this general methodology to design implicit and semi-implicit numerical schemes with the well-balanced property by means of well-balanced reconstructions operators. To our knowledge, no general framework to design this type of schemes has been developed so far. In the case of low-order schemes, some works can be found for finite-volume, finite-difference, finite-element and discontinuous Galerkin methods which only work for particular steady states (mainly zero velocity steady states): see, for instance, \cite{ChalonsKokh2020,Frank2020,Zhao2014,CASULLI199056,Vater2018ASM,BONAVENTURA2018209,Lee2019}. See \cite{CCM18} for well-balanced methods for all the one-dimensional steady states for the shallow-water system. 
Concerning high-order schemes, finite-difference, discontinuous Galerkin methods and combination of finite-volume/finite-element and finite-volume/finite-difference methods that are well-balanced for particular stationary solutions (mainly zero velocity steady states) are presented, for example, in \cite{busto2022,Huang2021HighOW,Tavelli2014,DUMBSER20138057,Giraldo2010}.
 
 In the context of semi-implicit numerical schemes, RK-IMEX setting  represents a powerful tool for the time discretization of system of the form (\ref{PDE_generalproblem}) if such it contains both stiff and non stiff terms. 
 Typical examples are hyperbolic systems with stiff hyperbolic or parabolic relaxation characterized by a relaxation parameter $\varepsilon$, \cite{pareschi2005implicit, boscarino2009class,boscarino2013flux, boscarino2013implicit, boscarino2014high}.
 
 In the hyperbolic to hyperbolic relaxation (HHR) a natural treatment consists in adopting RK-IMEX schemes, in which the relaxation term is treated by an implicit scheme, while the hyperbolic part is treated explicitly \cite{pareschi2005implicit, boscarino2009class}. On the other hand, in hyperbolic systems with parabolic relaxation (HPR), standard RK-IMEX schemes developed for HHR systems are not appropriate, because the characteristic speeds of the hyperbolic term diverge as the relaxation parameter vanishes. Slightly modified IMEX schemes become consistent discretization of the limit diffusive relaxed system, but suffer from the parabolic CFL restriction. In \cite{boscarino2013flux, boscarino2013implicit, boscarino2014high} this drawback has been overcome by a penalization method, so that the limit scheme becomes an implicit scheme for the limit diffusive relaxed system.
 
 Furthermore, in \cite{boscarino2017unified} a unified RK-IMEX approach has been introduced for systems which may admit both hyperbolic and parabolic relaxation (for example with space dependent relaxation parameter). Thus, this latter approach applies to both HHR and HPR. 
 All these approaches are capable to capture the correct asymptotic limit of the system when $ \varepsilon \to 0$, i.e., the schemes are asymptotic preserving (AP) independently of the scaling used.

The organization of the article is as follows: in Section 2 we present an overview to build both exactly well-balanced and well-balanced reconstruction operators and its application to design explicit high-order numerical schemes satisfying this property. Section 3 is devoted to obtain well-balanced implicit methods. We introduce a general procedure to design well-balanced reconstruction operators adapted to implicit methods and a general result is stated showing that well-balanced reconstruction operators lead to well-balanced methods.  Section 4 is focused on how the time discretization is performed, including the particular case of implicit first- and second-order well-balanced schemes.
Section 4 ends with the semi-implicit case. In Section 5, numerous numerical tests are performed to check well-balanced property of the implicit and semi-implicit numerical methods. Even though the introduced strategy can be used to design arbitrary high-order well-balanced methods, only first- and second-order methods have been implemented. We consider numerical tests for scalar and systems of balance laws: the transport equation, the Burgers equation with a non-linear source term and the shallow water model with and without Manning friction. Eventually, some conclusions are drawn in Section 6 and possible forthcoming works are also discussed.
 
\section{Preliminaries}
%

As discussed earlier, in previous works, two of the authors introduced a general procedure to design explicit methods satisfying these properties based on the use of reconstructions operators. This general strategy involves nonlinear problems to be solve at every computational cell and time step consisting in finding a stationary solution of \eqref{PDE_generalproblem} with given average in the cell. If the expression of the stationary solutions if available, exactly well-balanced methods can be designed. If it is not the case, one can obtain methods that are well-balanced.

\subsection{Exactly well-balanced methods}
Remember that, given the set $\{ \overline{u}_i \}$ of cell-averages of a function $u$
$$
\overline{u}_i = \frac{1}{\Delta x} \int_{x_{i-1/2}}^{x_{i+1/2}} u(x) \,dx,
$$
or their approximations using  a quadrature formula
$$
\overline{u}_i = \sum_{m=1}^M \alpha_m u(x^m_i),
$$
where $x_i^m$ and $\alpha_m$ are, respectively, the nodes and the weights of the quadrature formula, a reconstruction operator provides approximations of $u$ at the cells 
$$ u(x) \approx P_i(x; \{\bar u_j\}_{j \in \mathcal{S}_i} \}, \quad x \in [x_{i-1/2}, x_{i+1/2}],$$
where $x_{i \pm 1/2}$ represent the inter-cells.
Here for simplicity we assume that space discretization is uniform, so that the weights do not depend on $i$, and $x_i^m = x_{i-1/2}+c_m \Delta x$, $m=1,\ldots,M$, where $c_m$ denotes the nodes of the quadrature formula in the interval $[0,1]$. These approximations are obtained by interpolation or approximation techniques from the cell values at the cells belonging to the stencil of the $i$th cell:  $\mathcal{S}_i$ represents the set of their indexes. MUSCL, ENO, or CWENO are examples of high-order reconstruction operators. It will be assumed from now on that all the cell-averages are approximated using a fixed quadrature formula.

The design of explicit high-order well-balanced numerical methods  discussed in \cite{CastroPares2019} is based on the use of reconstruction operators that are well-balanced according to the following definition: 

\begin{Definition}
Given a stationary solution $u^e$ of \eqref{PDE_generalproblem}, a reconstruction operator $P_i(x)$ is said to be exactly well-balanced for $u^e$ if
$$
P_i(x; \{ \overline{u}^e_j \}_{j \in \mathcal{S}_i })=u^e(x), \quad  \forall x \in [x_{i-\frac{1}{2}}, x_{i+\frac{1}{2}}], \, \forall i.
$$
where $\overline{u}^e_j$ represent the exact cell-averages or the approximate cell-averages obtained by a quadrature formula from the stationary solution $u^e$.
\end{Definition}

Since, in general, a standard reconstruction operator is not expected to be  well-balanced, the following strategy introduced in \cite{sinum2008} is used to obtain a well-balanced operator $P_i$ from a standard one $Q_i$:
\begin{Algorithm}\label{alg:ewbnirec} 
Given a family of cell values $\{ \overline{u}_i \}$, at every cell $I_i=[x_{i-1/2}, x_{i+1/2}]$:

\begin{enumerate}

\item Find, if it is possible, a stationary solution $u_i^{e}(x)$ of \eqref{PDE_generalproblem} defined in the stencil of cell $I_i$  such that: 
$$
\sum_{m=1}^M  \alpha_m  u_i^{e} (x^m_i)  =  \overline{u}_i.
$$

\item  Apply the reconstruction operator to the cell values $\{v_j \}_{j \in \mathcal{S}_i}$ given by
$$
v_j = \overline{u}_j - \sum_{m=1}^M  \alpha_m  u_i^{e} (x^m_j)
$$
to obtain
$$
Q_i(x) = Q_i(x; \{v_j^n \}_{j \in \mathcal{S}_i }).
$$

\item Define
\begin{equation*} \label{step3}
P_i(x)=u_i^{e}(x)+Q_i(x).
\end{equation*}

\end{enumerate}
\end{Algorithm}

It can be easily proved that the reconstruction operator $P_i$ is exactly well-balanced provided that $Q_i$ is exact for the zero function.
Notice that, at every cell, a nonlinear problem has to be solved in the first step consisting in finding a stationary solution in the stencil of the cell with prescribed average in the cell. Therefore, the expression of the stationary solutions either in explicit or implicit form is required. 

Once the exactly well-balanced reconstruction operator has been built, the semi-discrete numerical methods writes as follows:

    \begin{equation}\label{EWBmeth}
\begin{split}     \frac{d u_i}{dt} &= -\frac{1}{ \Delta x} \left( F_{i+1/2}(t) - F_{i-1/2}(t) \right) +
\frac{1}{\Delta x} \left( f\left(u^{e,t}_i(x_{i+1/2}))\right)- f\left(u^{e,t}_i(x_{i-1/2}))\right) \right)\\
&  +  \sum_{m=1}^{M}\alpha_m \left( S(P_i^{t}(x_i^m))-S(u^{e,t}_i(x_i^m) \right)H_x(x_i^m), \,\forall i,
\end{split}
    \end{equation}
where
\begin{itemize} 

\item $\displaystyle P_i^t$ is the well-balanced reconstruction obtained from the cell values $\{u_i(t) \}$;

\item $\displaystyle u_i^{e,t}$ is the stationary solution found at the first step of the reconstruction procedure at the $i$th cell; 

\item  $\displaystyle F_{i+1/2}(t) = \mathbb{F}(u_{i+1/2}^{t, -}, u_{i+1/2}^{t, +}),$
where $\mathbb{F}$ is any consistent  numerical flux and
$$u_{i+1/2}^{t, -} = P_i^t(x_{i+1/2}), \quad u_{i+1/2}^{t, +} = P_{i+1}^t(x_{i+1/2}).$$

\end{itemize}

Notice that in order to use a lighter notation we write $u_i(t)$ in place of $\overline{u}_i(t)$ to denote a semidiscrete approximation of cell average of the solution, $$\frac{1}{\Delta x}\int_{x_{i-1/2}}^{x_{1+1/2}}u(x,t)\,dx.$$

This numerical method is well-balanced in the sense that, given any stationary solution $u^e$, the set of cell values $\{ \overline{u}^e_i \}$ is an equilibrium of the ODE system \eqref{EWBmeth}: see \cite{CastroPares2019}.

\subsection{Well-balanced methods}
When the expression of the stationary solutions in explicit or implicit form is not available, reconstruction operators that are just well-balanced but not exactly well-balanced can be designed following the idea developed in \cite{GomezCastroPares2020,gomez2021well,gomez2021collocation}: consider a numerical solver of the ODE system \eqref{ODE_stationarysolutions} that provides approximations of a stationary solution $u^e$ at the inter-cells and the quadrature points
$$
u^e_{i + 1/2} \approx u^e(x_{i+1/2}), \quad u^{e}_{i,m} \approx u^e(x^m_i), \ m= 1, \dots, M, \quad \forall i.
$$
Well-balanced reconstruction operators are then defined as follows:

\begin{Definition}
Given a stationary solution $u^e$ of \eqref{PDE_generalproblem}, a reconstruction operator $P_i(x)$ is said to be well-balanced for $u^e$ if, for every $i$:
\begin{eqnarray*}
& & P_i(x_{i \pm 1/2}; \{ \overline{u}^e_j \}_{j \in \mathcal{S}_i })=u^e_{i\pm 1/2},\\ & &  P_i(x_i^m; \{ \overline{u}^e_j \}_{j \in \mathcal{S}_i })=u^{e}_{i,m}, \ m = 1, \dots, M,
\end{eqnarray*}
where here the cell-averages $ \{ \overline{u}^e_i\}$ are given by
\begin{equation*}
    \overline{u}^e_i=\sum_{m=1}^M \alpha_m u_{i,m}^e.
\end{equation*}
\end{Definition}

Algorithm \ref{alg:ewbnirec} is modified as follows to obtain well-balanced reconstruction operators:
\begin{Algorithm}\label{alg:wbnirec} 
Given a family of cell values $\{ \overline{u}_i \}$, at every cell $I_i=[x_{i-1/2}, x_{i+1/2}]$:

\begin{enumerate}

\item Find, if it is possible, a stationary solution $u_i^{e}(x)$ of \eqref{PDE_generalproblem} defined in the stencil of cell $I_i$  such that: 
$$
\sum_{m=1}^M  \alpha_m  u_{i;i,m}^{e} =  \overline{u}_i,
$$
where
$$
 u_{i;i,m}^{e}  \approx u^e_i(x^m_i), \ m= 1, \dots, M,
$$
are the approximations provided by the selected numerical solver for \eqref{ODE_stationarysolutions} .

\item  Apply the reconstruction operator to the cell values $\{v_j \}_{j \in \mathcal{S}_i}$ given by
$$
v_j = \overline{u}_j - \sum_{m=1}^M  \alpha_m  u_{i;j,m}^{e}
$$
to obtain
$$
Q_i(x) = Q_i(x; \{v_j^n \}_{j \in \mathcal{S}_i }),
$$
where
$$
 u_{i;j,m}^{e}  \approx u^e_i(x^m_j), \ m= 1, \dots, M, \quad j \in \mathcal{S}_i.
$$

\item Define
\begin{eqnarray*} 
u_{i-1/2}^+ & = & u_{i;i-1/2}^e + Q_i(x_{i-1/2}),\\
u_{i+1/2}^- & = & u_{i;i+1/2}^e + Q_i(x_{i+1/2}), \\
P_i(x^m_i) & = & u_{i;i,m}^e +Q_i(x_i^m),  \ m=1, \dots,M,
\end{eqnarray*}
where
$$
 u_{i; i \pm 1/2}^{e}  \approx u^e_i(x_{i\pm 1/2}).
$$

\end{enumerate}
\end{Algorithm}

Observe that, in order to implement \eqref{EWBmeth} it is enough to compute the reconstructions at the inter-cells and at the quadrature points. 

The first step of the reconstruction procedure consists now of applying a numerical solver to the ODE system \eqref{ODE_stationarysolutions}  to find a solution with prescribed average at a cell. Two different strategies have been developed to solve these problems:
\begin{itemize}
    \item A control-based approach (see \cite{GomezCastroPares2020,gomez2021well}). The nonlinear problems to be solved in the reconstruction procedure are interpreted as control problems, in which the control is the value of the solution at the left extreme point of the stencil. The gradient of the functional is computed on the basis of the adjoint problem. Different gradient-type methods and the Newton's method are applied to solve the control problems. 
  
    \item A technique based on RK collocation methods (see \cite{gomez2021collocation}). RK collocation methods are used to solve \eqref{ODE_stationarysolutions} with prescribed average in each cell $I_i$ and to extend the found stationary solution to the cells belonging to stencil $\mathcal{S}_i$.
\end{itemize}

Once the well-balanced reconstruction operator has been built, the semi-discrete numerical method writes as follows:

    \begin{equation}\label{WBmeth}
\begin{split}     \frac{d u_i}{dt} &= -\frac{1}{ \Delta x} \left( F_{i+1/2}(t) - F_{i-1/2}(t) \right) +
\frac{1}{\Delta x} \left( f\left(u^{e,t}_{i;i+1/2})\right)- f\left(u^{e,t}_{i;i-1/2})\right) \right)\\
&  +  \sum_{m=1}^{M}\alpha_m \left( S(P_i^{t}(x_i^m))-S(u_{i;i,m}^{e,t}) \right)H_x(x_i^m), \,\forall i,
\end{split}
    \end{equation}
where now  $\displaystyle u^{e,t}_{i;i \pm 1/2}$, $\displaystyle u^{e,t}_{i;i,m}$ are the approximations at the inter-cells $x_{i\pm 1/2}$ and the quadrature points $x_i^m$ of the stationary solution $u_i^{e,t}$ found at the first step of the reconstruction procedure at the $i$th cell. 

This numerical method is well-balanced in the sense that, given any stationary solution $u^e$, the set of cell values $\{ {u}^e_i \}$ given by
$$
u^e_i = \sum_{m = 1}^M \alpha_m u^{e}_{i,m}
$$
is an equilibrium of the ODE system \eqref{WBmeth}: see \cite{gomez2021collocation}.

Observe that \eqref{EWBmeth} is the particular case of \eqref{WBmeth} corresponding to the exact ODE solver
for system \eqref{ODE_stationarysolutions}, i.e.
\begin{eqnarray*}
& & u^{e,t}_{i; i \pm 1/2} = u^{e,t}_i(x_{i + 1/2}), \\
& & u^{e,t}_{i; j,m} = u^{e,t}(x_j^m), \quad m = 1, \dots, M.
\end{eqnarray*}
Therefore, only the expression of the method \eqref{WBmeth} will be considered without loss of generality. 

\subsection{Explicit methods}
Explicit high-order well-balanced numerical methods are then obtained by applying an ODE solver (usually a TVD RK method) to the ODE systems \eqref{EWBmeth} or \eqref{WBmeth}.

\section{Implicit methods}
Although in principle implicit high-order well-balanced methods can be obtained by applying implicit ODE solvers to \eqref{EWBmeth} or \eqref{WBmeth}, in practice the well-balanced reconstruction of the unknown solution $u^{n+1}$ would lead to complex nonlinear systems that may be costly to solve. To avoid this we look for a solution of the ODE system in $[t^n,t^{n+1}]$ of the form $u_i(t) = u^n_i + u_i^f(t)$, and adopt  besides the standard reconstruction operator  ${Q}$, a new reconstruction $\tilde{Q}$, which will act on the perturbation $u_i^f$ as described below. Once the approximations at time $t^n$, $\{ u_i^n \}$, have been computed, in order to update them we proceed as follows: 
\begin{itemize}
    \item First, the well-balanced reconstruction procedure is applied to $\{u_i^n\}$ to obtain:
    $$P_i^n(x) = u^{e,n}_i(x) + {Q}_i(x; \{v_j^n\}_{j \in \mathcal{S}_i}),$$
where   $u^{e,n}_i(x)$ is the stationary solution found at the first step of the reconstruction procedure at the $i$th cell. 
\item Next we consider the following ODE system in the time interval $[t^n, t^{n+1}]$:
\begin{equation}\label{CP_timefluc}
    \begin{split}     \frac{d u_i^f}{dt} &= -\frac{1}{ \Delta x} \left( F_{i+1/2}(t) - F_{i-1/2}(t) \right) +
    \frac{1}{\Delta x} \left( f\left(u^{e,n}_{i;i+1/2}\right)- f\left(u^{e,n}_{i;i-1/2}\right) \right)\\
&  +  \sum_{m=1}^{M}\alpha_m \left( S(P_i^{t}(x_i^m))-S(u_{i;i,m}^{e,n}) \right)H_x(x_i^m), \,\forall i,
    \end{split}
\end{equation}
with initial condition
$$
    u_i^f(t^n) = 0, \quad \forall i.
$$
Here 
\begin{equation}
P_i^{t} (x) = P_i^n(x) + \widetilde{Q}_i (x; \{ u^f_j \}_{j \in \widetilde{\mathcal{S}}_i}) ,   
\label{eq:P}
\end{equation}
and
\begin{equation}
 F_{i+1/2}(t) = \mathbb{F}(u_{i+1/2}^{t, -}, u_{i+1/2}^{t, +}),
\label{eq:F}
\end{equation}
where 
\begin{equation}
u_{i+1/2}^{t, -} = P_i^t(x_{i+1/2}), \quad u_{i+1/2}^{t, +} = P_{i+1}^t(x_{i+1/2}).
\label{eq:u}
\end{equation}

\item Define:
\begin{equation}\label{numericalmethod}
    u_i^{n+1}=u_i^n+u_i^f(t^{n+1}).
\end{equation}

\end{itemize}

Observe that, although
$$
u_i(t) = u_i^n + u_i^f(t), \quad t \in [t^n, t^{n+1}]
$$
formally solves \eqref{WBmeth}, the reconstruction $P_i^t$ is not the same as the one in the previous section: while there one had
$$P_i^{t} (x) = u_i^{e,t} (x)  + {Q}_i (x; \{ v_j \}_{j \in \widetilde{\mathcal{S}}_i}),$$
now,
$$
P_i^t(x) =  u^{e,n}_i(x) + {Q}_i(x; \{v_j^n\}_{j \in \mathcal{S}_i}) + \widetilde{Q}_i (x; \{ u^f_j(t) \}_{j \in \widetilde{\mathcal{S}}_i}).
$$
The main differences are the following:
\begin{itemize}
    \item to compute $P_i^t$ the stationary solution $u^{e,n}_i$ is used instead of $u^{e,t}_i$;
    \item the reconstruction operator $\widetilde{Q}_i$ will be in practice easier and cheaper to compute than $Q_i$: in particular, the smoothness indicators obtained to compute $Q_i$ at time $t^n$ may be used to compute $\widetilde Q_i$. We shall require that $\tilde{Q}_i$ maintains null states and its order of accuracy is $p$.
\end{itemize}
The description of the fully discrete schemes will be completed in the next section, by specifying how to solve the ODE system (\ref{CP_timefluc}-\ref{numericalmethod}).

Some properties of the numerical schemes do not depend on the detail of the particular scheme that is adopted for the numerical solution of system (\ref{CP_timefluc}-\ref{numericalmethod}), therefore we shall discuss them in this section.

\subsection{Well-balanced property}
Here we state two results concerning the well-balanced properties of the schemes described at the beginning of the section.
\begin{Theorem}
Given a stationary solution $u^e$ of \eqref{PDE_generalproblem}, let us suppose that, at every cell, at every time step, $P_i^n$ is a well-balanced reconstruction operator.
Then, the numerical method \eqref{numericalmethod} is well-balanced for $u^e$, in the sense that, if the initial condition is given by
$$
u^0_i = \sum_{m = 1}^M \alpha_m u^{e}_{i,m},
$$
then
$$
u^n_i = u^0_i
$$
for every $n$ and every $i$.
\end{Theorem}
The proof is straightforward: it is enough to check that at every cell, at every time step, $u_i^f(t)\equiv 0$ is the solution of the Cauchy problem \eqref{CP_timefluc}.
As a corollary we have:
\begin{Theorem}
Given a stationary solution $u^e$ of \eqref{PDE_generalproblem}, let us suppose that, at every cell, at every time step, $P_i^n$ is an exactly well-balanced reconstruction operator.
Then, the numerical method \eqref{numericalmethod} is exactly well-balanced for $u^e$, in the sense that, if the initial condition is given by
$$
u^0_i = \sum_{m = 1}^M \alpha_m u^{e}(x_{i}^{m}),
$$
then
$$
u^n_i = u^0_i
$$
for every $n$ and every $i$.
\end{Theorem}

\section{Time discretization}
This section is devoted to time discretization. If 
 system (\ref{CP_timefluc}-\ref{numericalmethod}) is not {\em stiff\/}, i.e.\ if accuracy and stability restrictions on the time step $\Delta t$ are of the same order of magnitude, then one can adopt explicit schemes. This is often the case for hyperbolic systems of balance laws if one is interested in resolving all the waves of the system, i.e.\ if all the signals transported by the various waves are not negligible, and if the time scales associated to the right hand side are not too small. In such cases one can adopt explicit schemes such as explicit Runge-Kutta of multistep methods. In particular, strongly stability preserving schemes are generally adopted for the numerical solution of hyperbolic systems of balance laws, since they prevent formation of spurious oscillations due to time discretization (see \cite{gottlieb2001strong}). The literature on well-balanced schemes based on explicit schemes is too vast to mention it here. We just recall the following review papers \cite{CastroPares2019,Roe87,bermudez1994upwind,GNVC,LR96, LR97,LeVeque98,Gosse00,Gosse01,Gosse02,GALLICE2002713,berthon2016,BERTHON2018284,Desveaux16,desveauxRipa2016,Tang04,PerSi1,Audusse,Bouchut,M3AS07,subsonic,Kappeli14,M2AN_2001__35_4_631_0,Lukacova,Kurganov07,Chertock18,Chandrashekar15,CASELLES200916,CHERTOCK201836,weno,sinum2008, Noelle06,Noelle07,Russo,BERBERICH2021104858,Klingenberg19,Xing06,XING2006b,PARES2021109880,venas,CANESTRELLI2010291,Chandrashekar17,guerrero2021,ricchiuto_stabilized_2009,abgrall_high-order_2014,ricchiuto_explicit_2015}.
 For this reason in this paper we do not consider WB schemes based on explicit time discretization methods.
 The choice of the method adopted for time discretization depends on the problem we want to solve.
 We shall consider three different situations: 
 \begin{enumerate}
     \item only a source term or part of it is stiff while the hyperbolic term is non stiff;
     \item the source or part of it and some part of the hyperbolic term is stiff;
     \item both the hyperbolic term and the source are stiff and require implicit solver.
 \end{enumerate}
 Conceptually, the simplest case is the third one, which requires an implicit treatment of both source and the hyperbolic term, so this is the case we start with. 
 In such case one could adopt an implicit scheme for the treatment of \eqref{CP_timefluc}. 
 Most commonly used implicit Runge-Kutta schemes are the so called diagonally implicit schemes (DIRK), and in particular \emph{singly diagonally implicit}
 which are described by the following 
Butcher tableau
  \begin{equation}\label{DIRK}
\begin{array}{c|ccccc}
\gamma & \gamma  & 0      & 0 &\dots & 0 \\
c_2    &  a_{2,1}& \gamma &  0 & \dots & 0\\
c_3    & a_{3,1} & a_{3,2} & \gamma & \dots & 0 \\
\vdots & \vdots  & \vdots & \vdots & \ddots & \vdots \\
1   & a_{s,1} & a_{s,2} & a_{s,3}& \dots & \gamma \\\hline
   & a_{s,1} & a_{s,2} & a_{s,3}& \dots & \gamma \\
\end{array}
 \end{equation}
 will be applied to solve the Cauchy problems satisfied by the time fluctuations $u_i^f(t)$:
\begin{eqnarray*}
u_i^{f,k} & = &  \Delta t \sum_{l = 1}^{k-1} a_{k,l} L^l_i + \Delta t \gamma L_i^k, \quad k = 1, \dots, s,\\
u_i^{f,n+1} & = & u_i^{f,s},
\end{eqnarray*}
where
    \begin{equation}
    \begin{split}
         L_i^k&= -\frac{1}{ \Delta x} \left( F_{i+1/2}^{k} - F_{i-1/2}^{k} \right) +
\frac{1}{ \Delta x} \left( f\left(u^{e,n}_{i;i+1/2}\right)- f\left(u^{e,n}_{i;i-1/2}\right) \right)\\
& +
 \sum_{m=1}^{M}\alpha_m \left( S(P_i^{k}(x_i^m))-S(u^{e,n}_{i;i,m}) \right) H_x(x_i^m), \,k=1,\dots, s,
    \end{split}
\end{equation}
with obvious notation, that can be written as well in the form:
\begin{eqnarray*}
u_i^{f,k} &=& \sum_{l=1}^{k-1} \gamma_{k,l} u_i^{f,l} + \gamma \Delta t L_i^{k},  \quad k = 1, \dots, s, \\
u_i^{f,n+1} & = & u_i^{f,s},
\end{eqnarray*}
for some coefficients $\gamma_{k,l}$.
The choice of the particular DIRK scheme depends on the problem. If the systems we want to solve are very stiff, then it is adisable to adopt an $L$-stable scheme, which is the type of schemes we shall use in this paper. 
\subsection{First-order schemes}
The backward Euler method is used to discretize in time. 
The system for the time fluctuations writes thus as follows:
\begin{equation}\label{timefluc_O1}
\begin{split}
u_i^{f,n+1} & = -\frac{\Delta t}{ \Delta x} \left( F_{i+1/2}^{n+1} - F_{i-1/2}^{n+1} \right) +
\frac{1}{ \Delta x} \left( f\left(u^{e,n}_{i;i+1/2}\right)- f\left(u^{e,n}_{i;i-1/2}\right) \right)\\
& +\Delta t \left( S(P_i^{n+1}(x_i))-S(u^{e,n}_{i;i}) \right)H_x(x_i),  
\end{split}
\end{equation}
where the midpoint rule has been used to approximate the integrals. 
The reconstruction operators $Q_i$ and $\widetilde Q_i$ are the trivial piecewise constant ones so that:
$$
P_i^n(x) =  u^{e,n}_i(x), \quad P_i^{n+1} = u^{e,n}_i(x) + u_i^{f,n+1},
$$
and therefore the reconstructed states are defined as follows:
\begin{equation*}
    \begin{cases}
    \displaystyle u_{i+\frac{1}{2}}^{n+1, -} = u^{e,n}_{i;i+\frac{1}{2}}+u_i^{f,n+1}= u_{i+\frac{1}{2}}^{n,-}+u_i^{f,n+1},\\
    \displaystyle u_{i+\frac{1}{2}}^{n+1, +}=u^{e,n}_{i+1;i+\frac{1}{2}}+u_{i+1}^{f,n+1}=u_{i+\frac{1}{2}}^{n,+}+u_{i+1}^{f,n+1}.
    \end{cases}
\end{equation*}
Once System \eqref{timefluc_O1} has been solved, the cell-averages are updated as follows
\begin{equation*}
  u_i^{n+1}=u_i^n+u_i^{f,n+1}. 
\end{equation*}

Notice that, except for the case of a linear problem, the system to be solved to compute the time fluctuations $u_i^{f,n+1}$, \eqref{timefluc_O1}, is nonlinear: a numerical method such as a fixed-point algorithm or the Newton's method will be applied to solve them.

\subsection{Second-order schemes}
The second-order implicit Runge-Kutta method whose Butcher tableau is
  \begin{equation}\label{RKcol2}
\begin{array}{c|cc}
\gamma & \gamma & 0 \\
1 & 1-\gamma & \gamma \\\hline
& 1-\gamma & \gamma,\\
\end{array}
 \end{equation}
where $\gamma=1-\frac{1}{\sqrt{2}}$, will be used now for the time discretization.
Since the scheme is \emph{stiffly accurate}, i.e.\ $a_{s,i} = b_i,i=1,\ldots,s$, then the numerical solution coincides with the last stage value. 
Then, the fully discrete numerical method is as follows:
\begin{equation} \label{timefluc_O2}
\begin{split}
    u_i^{f,1}&= \Delta t\gamma L_i^1,\\
    u_i^{f,2}&= \Delta t (1-\gamma) L_i^1+ \Delta t \gamma L_i^2,\\
    u_i^{f,n+1}&=u_i^{f,2},
\end{split}
\end{equation}
where 
    \begin{equation}
    \begin{split}
         L_i^k&= -\frac{1}{ \Delta x} \left( F_{i+1/2}^{k} - F_{i-1/2}^{k} \right) +
\frac{1}{\Delta x} \left( f\left(u^{e,n}_{i;i+1/2}\right)- f\left(u^{e,n}_{i;i-1/2}\right) \right)\\
& +
\left( S(P_i^{k}(x_i))-S(u^{e,n}_{i;i}) \right)H_x(x_i), \,k=1,2.
    \end{split}
\end{equation}
Again, the midpoint rule is applied to approximate the integrals and $u^{e,n}_i(x)$ is the stationary solution such that $u_i^{e,n}(x_i)\approx u^{e,n}_{i;i}=u_i^n$.

Since
$$
L_i^1=\frac{u_i^{f,1}}{\Delta t \gamma},
$$
 \eqref{timefluc_O2} can be equivalently written as
\begin{equation} \label{metnum_O2_v2}
\begin{split}
    u_i^{f,1}&=\Delta t \gamma L_i^1,\\
    u_i^{f,n+1}&=\frac{(1-\gamma)}{\gamma}u_i^{f,1}+\Delta t \gamma L_i^{n+1},
\end{split}
\end{equation}
where the superindex $2$ has been replaced by $n + 1$, since $u_i^{f, n+1} = u_i^{f,2}$.

The well-balanced reconstruction operator  $P_i^n(x)$ is computed on the basis of a MUSCL-type reconstruction operator using three-point stencils
$$
\mathcal{S}_i = \{ i-1, i, i+1 \}
$$
as follows:
\begin{Algorithm} Given the approximations $\{ u_i^{n}\}$ at time $t^n$: 

\begin{enumerate}
\item   Find, if possible, a stationary solution $u^{e,n}_i(x)$ such that
$$
 u_{i;i}^{e,n} =  u_i^n,
$$
where
$$
 u_{i;i}^{e,n}  \approx u^e_i(x_i).
$$
\item Apply the second-order reconstruction operator $Q_i$ to $\{v_j^n\}_{j \in \mathcal{S}_i}$ given by
\begin{equation*}
v_j^n=  u_j^n -   u_{i;j}^{e,n} ,\quad j \in \mathcal{S}_i, 
\end{equation*}
where
$$
 u_{i;j}^{e,n}  \approx u^e_i(x_j), \,  \in \mathcal{S}_i,
$$
to obtain
\begin{equation*}
Q_i^n(x)=Q_i^n(x;\{v_j^n\}_{j \in \mathcal{S}_i})=\Delta_i v^n (x-x_i).
\end{equation*}
Here, $\Delta_i v^n$ is a slope limiter, such as the minmod limiter,
\begin{equation}
    \Delta_i v^n=\text{minmod}\left(\frac{v_{i+1}^n-v_i^n}{\Delta x}, \frac{v_{i}^n-v_{i-1}^n}{\Delta x}\right),
\end{equation}
where

\begin{equation}
\text{minmod}(a,b)= \left\{ \begin{array}{cl}

             \text{min}(a,b) &   if  \, a > 0, b > 0, \\ \text{max}(a,b) &   if  \, a < 0, b <0, \\
             0 & otherwise,

             \end{array}
             \right.
\end{equation}
or the avg limiter:
\begin{equation}
    \Delta_i v^n=\text{avg}\left(\frac{v_{i+1}^n-v_i^n}{\Delta x}, \frac{v_{i}^n-v_{i-1}^n}{\Delta x}\right),
\end{equation}
where

\begin{equation}
\text{avg}(a,b)= \left\{ \begin{array}{cl}

             \displaystyle\frac{|a|b+|b|a}{|a|+|b|} &   if  \, |a|+|b| > 0, \\ 
             0 & otherwise.

             \end{array}
             \right.
\end{equation}

\item Define
\begin{eqnarray*} 
u_{i-1/2}^{n,+} & = & u_{i;i-1/2}^{e,n} + Q^n_i(x_{i-1/2})= u_{i;i-1/2}^{e,n} - \frac{1}{2}\Delta_i v^n ,\\
u_{i+1/2}^{n,-} & = & u_{i;i+1/2}^{e,n} + Q^n_i(x_{i+1/2}) = u_{i;i+1/2}^{e,n}  +  \frac{1}{2}\Delta_i v^n , \\
P^n_i(x_i) & = & u_{i;i}^{e.n} +Q^n_i(x_i) = u_{i;i}^{e,n}, 
\end{eqnarray*}
where
$$
 u_{i; i \pm 1/2}^{e,n}  \approx u^e_i(x_{i\pm 1/2}).
$$

\end{enumerate}
\end{Algorithm}



Finally, two different choices for  $\widetilde{Q}^k_i$ are considered: the trivial piecewise constant reconstruction and  a piecewise linear one that uses the slope limiters of $P_i^n$.

\subsubsection{Piecewise constant reconstruction}\label{O2_3points}
In this case, $\widetilde{\mathcal{S}}_i = \{i\}$ and the reconstruction operator is given by
$$ \widetilde{Q}^k_i(x) =  u_i^{f,k}.$$
With this definition one has:
\begin{eqnarray*}
u^{k,-}_{i + 1/2} &=& u_{i+1/2}^{n,-} + \widetilde{Q}^k_i(x_{i-1/2}) = u_{i+1/2}^{n,-} + u_i^{f,k},\\
u^{k,+}_{i - 1/2} &=& u_{i-1/2}^{n,+} +  + \widetilde{Q}^k_i(x_{i+1/2})  = u_{i-1/2}^{n,+} +  u_i^{f,k},\\
P_i^k(x_i) &=& P_i^n(x_i) + \widetilde{Q}^k_i(x_i) =  u_{i;i}^{e,n} + u_i^{f,k}.
\end{eqnarray*}

\begin{Theorem}
Let us suppose that $\widetilde Q_i$ is the piecewise constant reconstruction operator. Then 
$$P_i^t(x)=P_i^n(x)+\widetilde Q_i$$
is a second-order reconstruction operator.
\end{Theorem}
\begin{proof}
Given a function $u(x, t)$, we consider the reconstruction operator
$$P_i^t(x)=P_i^n(x)+\overline{u}_i(t)-\overline{u}_i^n,$$
where $\overline{u}_i(t)$ and $\overline{u}_i^n$ represent the cell-averages of $u$ at the $i$-th cell at times $t$ and $t^n$, respectively, and $P_i^n$ is a second-order well-balanced reconstruction operator applied to $\{\overline{u}_i^n\}$.

Let us see that $P_i^t$ is a second-order reconstruction operator: given $x \in I_i$
and assuming that $\Delta t = O(\Delta x)$ we have
\begin{eqnarray*}
P_i^t(x) &  =&   P_i^n(x) + \overline{u}_i(t) - \overline{u}_i^n\\
& = & u(x, t^n) +   u(x_i, t) - u(x_i, t^n) + O(\Delta x^2) \\
& = & u(x,  t^n)   + u(x,  t^n)  +  \partial_x u(x,  t^n) (x_i-x)+ \partial_t  u(x,  t^n) (t - t^n) \\
&   &  -  u(x,  t^n)  - \partial_x u(x,  t^n) (x_i-x) + O(\Delta x^2) \\
&  = &   u(x,  t^n)   +   \partial_t  u(x,  t^n) (t -  t^n) +  O(\Delta x^2) \\
& =  & u(x, t) + O(\Delta x^2).
\end{eqnarray*}
Then, $P_i^t$ is second-order accurate.
\end{proof}

\subsubsection{Piecewise linear reconstruction operator}\label{O2_5points}
In this case, $\widetilde{\mathcal{S}}_i = \mathcal{S}_i = \{i-1,i, i+1\}$ and the reconstruction operator is given by
$$
\widetilde{Q}^k_i(x) = u_i^{f,k} + {\Delta}u^{f,k}_i (x- x_i),
$$
where
$$
{\Delta}u^{f,k}_i = \varphi_i^{n,L}\frac{\left(u_i^{f,k} -u_{i-1}^{f,k}\right)}{\Delta x}+ \varphi_i^{n,R}\frac{\left(u_{i+1}^{f,k}-u_{i}^{f,k}\right)}{\Delta x}.
$$
Here, $\varphi_i^{n,L}$ and $\varphi_i^{n,R}$ are two slope limiters computed using the approximations $u_i^n$ at time $t^n$. For instance, if the avg limiter is chosen, one has:
\begin{equation}
    \varphi_i^{n,L}=\frac{|d_R|}{|d_L|+|d_R|}, \quad \varphi_i^{n,R}=\frac{|d_L|}{|d_L|+|d_R|},
\end{equation}
where
\begin{equation}
    d_L=u_i^n-u_{i-1}^n, \quad d_R=u_{i+1}^n-u_i^n.
\end{equation}
In this case we have:
With this definition one has:
\begin{eqnarray*}
u^{k,-}_{i + 1/2} &=& u_{i+1/2}^{n,-} + \widetilde{Q}^k_i(x_{i-1/2}),\\
u^{k,+}_{i - 1/2} &=& u_{i-1/2}^{n,+} +  + \widetilde{Q}^k_i(x_{i+1/2}),\\
P_i^k(x_i) &=& P_i^n(x_i) + \widetilde{Q}^k_i(x_i) =  u_{i;i}^{e,n} + u_i^{f,k},
\end{eqnarray*}
where
\begin{eqnarray}
\widetilde Q_i^k(x_{i-1/2}) &  = & u_i^{f,k} -  \frac{1}{2}\varphi_i^{n,L}\left(u_{i}^{f,k}-u_{i-1}^{f,k}\right)-\frac{1}{2}\varphi_i^{n,R}\left(u_{i+1}^{f,k}-u_{i}^{f,k}\right),\\
\widetilde Q_i^k(x_{i+1/2}) &  = & u_i^{f,k} +  \frac{1}{2}\varphi_i^{n,L}\left(u_{i}^{f,k}-u_{i-1}^{f,k}\right)+\frac{1}{2}\varphi_i^{n,R}\left(u_{i+1}^{f,k}-u_{i}^{f,k}\right).
\end{eqnarray}



\subsection{Semi-implicit methods}
If not all terms of the equation are stiff, then it is not necessary to use a fully implicit scheme for the whole system.

Let us suppose, for example, that the problem writes as follows:
\begin{equation}\label{PDE_generalproblem_semi}
    u_t+f^1(u)_x + f^2(u)_x =S^1(u)H_{x} + S^2(u) ,
\end{equation}
where $H$ is a known function, 
with $f^1$ and $S^1$ non stiff, and $f^2$ and $S^2$ stiff. Then the problem can be more efficiently treated by adopting IMEX methods, in which the non stiff terms are treated explicitly, while the stiff terms are treated implicitly.
We select numerical fluxes $F^i(u_l, u_r)$, $i =1,2$ consistent with $f^i$, $i = 1,2$ and  an IMEX method with Butcher tableaux:
  \begin{equation}\label{DIRK_semi}
  \begin{array}{c|ccccc}
0 & 0 & 0      & 0 &\dots & 0 \\
\tilde c_2    &  \tilde a_{2,1}& 0 &  0 & \dots & 0\\
\tilde c_3    & \tilde a_{3,1} & \tilde a_{3,2} & 0 & \dots & 0 \\
\vdots & \vdots  & \vdots & \vdots & \ddots & \vdots \\
\tilde c_s   & \tilde a_{s,1} & \tilde a_{s,2} & \tilde a_{s,3}& \dots & 0 \\\hline
   & \tilde b_1 & \tilde b_{2} & \tilde b_{3}& \dots & \tilde b_s \\
\end{array}, \qquad
\begin{array}{c|ccccc}
\gamma & \gamma  & 0      & 0 &\dots & 0 \\
c_2    &  a_{2,1}& \gamma &  0 & \dots & 0\\
c_3    & a_{3,1} & a_{3,2} & \gamma & \dots & 0 \\
\vdots & \vdots  & \vdots & \vdots & \ddots & \vdots \\
1   & a_{s,1} & a_{s,2} & a_{s,3}& \dots & \gamma \\\hline
   & a_{s,1} & a_{s,2} & a_{s,3}& \dots & \gamma \\
\end{array}
 \end{equation}
 The numerical method writes then as follows:
 \begin{eqnarray*}
u_i^{f,k} & = &  \Delta t \sum_{l = 1}^{k-1} \tilde a_{k,l} L^{1,l}_i + 
\Delta t \sum_{l = 1}^{k-1} a_{k,l} L^{2,l}_i + \Delta t \gamma L_i^{2,k}, \quad k = 1, \dots, s,\\
u_i^{f,n+1} & = & \Delta t \sum_{l = 1}^s \tilde b_i L^{1,l}_i + 
\Delta t \sum_{l = 1}^{s-1} a_{s,i} L^{2,l}_i + \Delta t\gamma L^{2,s}_i,
\end{eqnarray*}
where
    \begin{equation}
    \begin{split}
         L_{i}^{1,k}&= -\frac{1}{\Delta x} \left( F_{i+1/2}^{1,k} - F_{i-1/2}^{1,k} \right) +
\frac{1}{ \Delta x} \left( f^1\left(u^{e,n}_{i;i+1/2}\right)- f^1\left(u^{e,n}_{i;i-1/2}\right) \right)\\
& +
\sum_{m=1}^{M}\alpha_m  \left( S^1(P_i^{k}(x_i^m))-S^1(u^{e,n}_{i;i,m}) \right)H_{x}(x_i^m);\\
L_{i}^{2,k}&= -\frac{1}{ \Delta x} \left( F_{i+1/2}^{2,k} - F_{i-1/2}^{2,k} \right) +
\frac{1}{ \Delta x} \left( f^2\left(u^{e,n}_{i;i+1/2}\right)- f^2\left(u^{e,n}_{i;i-1/2}\right) \right)\\
& +
\sum_{m=1}^{M}\alpha_m  \left( S^2(P_i^{k}(x_i^m))-S^2(u^{e,n}_{i;i,m}) \right);
    \end{split}
\end{equation}
for  $k=1,\dots, s$.
For instance, if the second order IMEX method with tableaux
  \begin{equation}\label{DIRK_O2}
  \begin{array}{c|cc}
0 & 0 & 0\\
\frac{1}{2 \gamma} & \frac{1}{2 \gamma}  & 0 \\ \hline
   & 1 - \gamma & \gamma\\
\end{array}, \qquad
\begin{array}{c|cc}
\gamma & \gamma  & 0 \\
1    &  1 -\gamma & \gamma \\\hline
   & 1 - \gamma & \gamma\\
\end{array},
 \end{equation}
 with
 $\gamma = (2 - \sqrt{2})/2$ is selected, the numerical method writes as follows:
 \begin{eqnarray*}
 u_i^{f,1} &= & \Delta t \gamma L_i^{2,1}, \\
 u_i^{f,2} & = & \frac{\Delta t}{2 \gamma } L_i^{1,1} + \Delta t (1 -  \gamma) L_i^{2,1} + \Delta t \gamma L_i^{2,2},\\
 u_i^{f, n+1} &=& {\Delta t}\left( (1 - \gamma)L_i^{1,1} + \gamma L_i^{1,2} + 
 (1- \gamma)L_i^{2,1} + \gamma L_i^{2,2} \right),
 \end{eqnarray*}
 or, equivalently,
 \begin{eqnarray*}
 u_i^{f,1} &= & \Delta t \gamma L_i^{2,1}, \\
 u_i^{f,2} & = & \frac{\Delta t}{2 \gamma }  L_i^{1,1} + \frac{1 - \gamma}{\gamma} u_i^{f,1} + \Delta t \gamma L_i^{2,2},\\
 u_i^{f, n+1} &= & u_i^{f,2} + \Delta t\left( 1 - \gamma - \frac{1}{2 \gamma} \right) L_i^{1,1} + \gamma \Delta t L_i^{1,2}.
 \end{eqnarray*}

\section{Numerical tests}

 The following acronyms will be used in this section to denote the different methods considered:

\begin{itemize}

\item EXWBM$p$: explicit well-balanced numerical method of order $p$ where the well-balanced reconstruction operator is based on RK collocation methods.

\item IEWBM$p$: implicit exactly well-balanced numerical method of order $p$.

\item IWBM$p$: implicit well-balanced numerical method of order $p$ where the well-balanced reconstruction operator is based on RK collocation methods.

\item SIEWBM$p$: semi-implicit exactly well-balanced numerical method of order $p$.

\item SIWBM$p$: semi-implicit well-balanced numerical method of order $p$ where the well-balanced reconstruction operator is based on RK collocation methods.

\end{itemize}

Although in principle the methodology introduced here can be used to design high-order well-balanced implicit or semi-implicit numerical methods, we have only implemented so far first- and second-order methods. The midpoint rule is considered to approximate the integrals and 1-stage RK collocation methods are applied to obtain the discrete stationary solutions and to solve the local nonlinear problems in the first step of the well-balanced reconstruction procedure.

\subsection{Transport equation}
 Let us consider the linear transport equation
\begin{equation}\label{eq_transport}
  u_t+cu_x= \alpha u,  
\end{equation}
where  $c, \alpha \in \mathbb{R}$. The stationary solutions of \eqref{eq_transport} are the 1-parameter family:
$$
u^e(x)=C e^{\frac{\alpha}{c} x}, \quad C \in \mathbb{R}.
$$
Therefore, given a family of cell averages 
$\{u_i^n\}$, and considering that cell-averages and pointwise values of a smooth function at cell center agree to second-order in $\Delta x$, the stationary solution $u^{e,n}_i(x)$ such that $u_i^{e,n}(x_i)=u_i^n$ is
$$
u^{e,n}_i(x)=u_i^n e^{\frac{\alpha}{c} (x-x_i)}.
$$
Exactly well-balanced methods will be considered. The Rusanov numerical flux
 $$
 \mathbb{F}(u,v)=\frac{c}{2}(u+v)-\frac{k}{2}(v-u),
 $$
 is considered, with $k=|c|$.
 
 \subsubsection{First-order method}
In the case of the  first-order scheme, one has
 \begin{equation*}
     \begin{split}
     F_{i+1/2}^{n+1}&=\frac{c}{2}\left(u_i^{e,n}(x_{i+1/2})+u_i^{f,n+1}+u_{i+1}^{e,n}(x_{i+1/2})+u_{i+1}^{f,n+1}\right)\\ &-\frac{k}{2} \left(u_{i+1}^{e,n}(x_{i+1/2})+u_{i+1}^{f,n+1}-u_i^{e,n}(x_{i+1/2})-u_i^{f,n+1}\right),
      \end{split}
 \end{equation*}
 which leads to the following expression of \eqref{timefluc_O1}:
  \begin{equation*} \label{metnum_o1_transport}
\begin{split}
u_i^{f,n+1} & = - \frac{\Delta t}{\Delta x} \left( \frac{c}{2} \left(u_{i+1}^{f,n+1}-u_{i-1}^{f,n+1}\right)-\frac{k}{2}\left(u_{i+1}^{f,n+1}-2u_{i}^{f,n+1}+u_{i-1}^{f,n+1}\right)\right)\\
 &  -\frac{\Delta t}{ \Delta x} \left(  \mathbb{F}\left(u_i^{e,n}(x_{i+1/2}),u_{i+1}^{e,n}(x_{i+1/2})\right) - \mathbb{F}\left(u_{i-1}^{e,n}(x_{i-1/2}),u_i^{e,n}(x_{i-1/2})\right)\right)\\
 &+\frac{\Delta t}{ \Delta x} \left( f\left(u_i^{e,n}(x_{i+1/2})\right)-f\left(u_i^{e,n}(x_{i-1/2})\right)\right)
 +\alpha \Delta t \left(u_i^{e,n}(x_i)+u_i^{f,n+1}-u_i^{e,n}(x_i)\right).
\end{split}
\end{equation*}
Since
$$
\mathbb{F}\left(u_i^{e,n}(x_{i+1/2}),u_{i+1}^{e,n}(x_{i+1/2})\right)=\mathbb{F}\left(u^{n,-}_{i+1/2},u^{n,+}_{i+1/2}\right)
= F_{i+1/2}^n,
$$
the linear system for the time fluctuations for the first-order case is:
  \begin{equation*} 
\begin{split}
u_i^{f,n+1} + \frac{\Delta t}{ \Delta x} \left( \frac{c}{2} \left(u_{i+1}^{f,n+1}-u_{i-1}^{f,n+1}\right)-\frac{k}{2}\left(u_{i+1}^{f,n+1}-2u_{i}^{f,n+1}+u_{i-1}^{f,n+1}\right)\right) - \alpha\Delta t u_i^{f,n+1}  = \\ 
- \frac{\Delta t}{ \Delta x} \left(F_{i+1/2}^{n}- F_{i-1/2}^{n} \right)  + \frac{\Delta t}{ \Delta x} \left( f\left(u_i^{e,n}(x_{i+1/2})\right)-f\left(u_i^{e,n}(x_{i-1/2})\right)\right).
\end{split}
\end{equation*}
If boundary conditions are neglected, a linear system has to be solved whose matrix is tridiagonal with coefficients
$$
d_0 = 1 + \lambda k - \alpha \Delta t, \quad  d_{-1}= - \frac{\lambda}{2} ( c+ k ), \quad d_1 = - \frac{\lambda}{2}( -c+ k ), 
$$
in the main, the lower, and the upper diagonals respectively, where
$$
\lambda = \frac{\Delta t}{\Delta x}.
$$

\subsubsection{Second-order methods}
The system of equations for $u_i^{f,1}$ is as follows:
\begin{equation*}\label{metnum_o2_transport_u1_v1}
    \begin{split}
        u_i^{f,1}&= -\frac{\gamma \Delta t}{ \Delta x} \left( F_{i+1/2}^{1} - F_{i-1/2}^{1}\right) 
        +
\frac{\gamma \Delta t}{ \Delta x} \left( f\left(u^{e,n}_i(x_{i+1/2})\right)- f\left(u^{e,n}_i(x_{i-1/2})\right) \right) \\
&+
\gamma \Delta t \left( u_i^{e,n}(x_i)+u_i^{f,1}-u_i^{e,n}(x_i) \right).
    \end{split}
\end{equation*}

If the piecewise constant reconstruction $\widetilde{Q}^k_i$ described in Subsection \ref{O2_3points} is selected, the expressions of the numerical flux is as follows:
\begin{equation*}
     \begin{split}
     F_{i+1/2}^{1}&=\frac{c}{2}\left(u_{i+1/2}^{n, -}+ u_i^{f,k} +u_{i+1/2}^{n, +}+ u_{i+1}^{f,k}\right)\\
     &-\frac{k}{2} \left(u_{i+1/2}^{n, +}+ u_{i+1}^{f,k} - u_{i+1/2}^{n, -}- u_i^{f,k}\right),
      \end{split}
 \end{equation*}
 and the expression of the numerical method is as follows:
 \begin{equation*} \label{metnum_o2_transport_u1_v3}
\begin{split}
u_i^{f,1} &+ \frac{\gamma \Delta t}{ \Delta x} \frac{c}{2}   \left(u_{i+1}^{f,1}-u_{i-1}^{f,1}\right)- \frac{\gamma \Delta t k}{2  \Delta x }    \left(u_{i+1}^{f,1}-2u_i^{f,1}+u_{i-1}^{f,1}\right)
- \gamma \alpha \Delta t u_i^{f,1}  = \\ 
&- \frac{\gamma \Delta t}{ \Delta x} \left(F_{i+1/2}^{n}- F_{i-1/2}^{n} \right)  + \frac{\gamma \Delta t}{ \Delta x} \left( f\left(u_i^{e,n}(x_{i+1/2})\right)-f\left(u_i^{e,n}(x_{i-1/2})\right)\right).
\end{split}
\end{equation*}
 \begin{equation*} \label{metnum_o2_transport_un+1_v2}
\begin{split}
u_i^{f,n+1} &+ \frac{\gamma \Delta t}{ \Delta x} \frac{c}{2}   \left(u_{i+1}^{f,n+1}-u_{i-1}^{f,n+1}\right)- \frac{\gamma \Delta t k}{2  \Delta x }    \left(u_{i+1}^{f,n+1}-2u_i^{f,n+1}+u_{i-1}^{f,n+1}\right)
- \gamma \alpha \Delta t u_i^{f,n+1}  = \\ 
&- \frac{\gamma \Delta t}{ \Delta x} \left(F_{i+1/2}^{n}- F_{i-1/2}^{n} \right)  + \frac{\gamma \Delta t}{ \Delta x} \left( f\left(u_i^{e,n}(x_{i+1/2})\right)-f\left(u_i^{e,n}(x_{i-1/2})\right)\right)+\frac{1-\gamma}{\gamma}u_i^{f,1}.
\end{split}
\end{equation*}
If, again, boundary conditions are neglected, two linear systems have to be solved with the same tridiagonal matrix whose coefficients are now
$$
d_0 = 1 + \gamma \lambda k - \gamma \alpha \Delta t, \quad  d_{-1}= - \frac{\gamma\lambda}{2} ( c+ k ), \quad d_1 = - \frac{\gamma\lambda}{2}( -c+ k ).
$$

If we consider now the piecewise linear reconstruction described in Subsection \ref{O2_5points}, the implementation of the numerical method \eqref{timefluc_O2} leads to solve the following linear systems with pentadiagonal matrices:
  \begin{equation*} \label{metnum_o2_transport_u1_5p}
\begin{array}{l}
 u_i^{f,1}  \left [ 1 -  \gamma \alpha \Delta t+ \displaystyle \frac{\gamma \Delta t}{ \Delta x} \frac{c}{2}\left(\varphi_i^{n,L}-\varphi_i^{n,R}+\frac{1}{2}\left(\varphi_{i+1}^{n,L}-\varphi_{i-1}^{n,R}\right)\right)+\frac{\gamma \Delta t}{ \Delta x}\frac{k}{2}\left(2-\frac{1}{2}\left(\varphi_{i+1}^{n,L}+\varphi_{i-1}^{n,R}\right)\right)\right]
  \\ 
 + u_{i+1}^{f,1}  \left [ \displaystyle \frac{\gamma \Delta t}{ \Delta x} \frac{c}{2}\left(1+\varphi_i^{n,R}+\frac{1}{2}\left(\varphi_{i+1}^{n,R}-\varphi_{i+1}^{n,L}\right)\right)-\frac{\gamma \Delta t}{ \Delta x}\frac{k}{2}\left(1+\frac{1}{2}\left(\varphi_{i+1}^{n,R}-\varphi_{i+1}^{n,L}\right)\right)\right]\\
  + u_{i-1}^{f,1}  \left [ \displaystyle \frac{\gamma \Delta t}{ \Delta x} \frac{c}{2}\left(-1-\varphi_i^{n,L}+\frac{1}{2}\left(\varphi_{i-1}^{n,R}-\varphi_{i-1}^{n,L}\right)\right)+\frac{\gamma \Delta t}{ \Delta x}\frac{k}{2}\left(-1+\frac{1}{2}\left(\varphi_{i-1}^{n,R}-\varphi_{i-1}^{n,L}\right)\right)\right]\\
  + u_{i+2}^{f,1}  \left [ \displaystyle \frac{\gamma \Delta t}{ \Delta x} \varphi_{i+1}^{n,R} \frac{k-c}{4}\right]
  + u_{i-2}^{f,1}  \left [ \displaystyle \frac{\gamma \Delta t}{ \Delta x} \varphi_{i-1}^{n,L} \frac{k+c}{4}\right]\\
 =
- \displaystyle \frac{\gamma \Delta t}{ \Delta x} \left(F_{i+1/2}^{n}- F_{i-1/2}^{n} \right)  + \displaystyle \frac{\gamma \Delta t}{ \Delta x} \left( f\left(u_i^{e,n}(x_{i+1/2})\right)-f\left(u_i^{e,n}(x_{i-1/2})\right)\right).  
\end{array}
\end{equation*}

  \begin{equation*} \label{metnum_o2_transport_un+1_5p}
\begin{array}{l}
 u_i^{f,n+1}  \left [ 1 - \displaystyle \gamma \alpha \Delta t+ \displaystyle \frac{\gamma \Delta t}{ \Delta x} \frac{c}{2}\left(\varphi_i^{n,L}-\varphi_i^{n,R}+\frac{1}{2}\left(\varphi_{i+1}^{n,L}-\varphi_{i-1}^{n,R}\right)\right)+\frac{\gamma \Delta t}{\Delta x}\frac{k}{2}\left(2-\frac{1}{2}\left(\varphi_{i+1}^{n,L}+\varphi_{i-1}^{n,R}\right)\right)\right]
  \\ 
 + u_{i+1}^{f,n+1}  \left [ \displaystyle \frac{\gamma \Delta t}{ \Delta x} \frac{c}{2}\left(1+\varphi_i^{n,R}+\frac{1}{2}\left(\varphi_{i+1}^{n,R}-\varphi_{i+1}^{n,L}\right)\right)-\frac{\gamma \Delta t}{ \Delta x}\frac{k}{2}\left(1+\frac{1}{2}\left(\varphi_{i+1}^{n,R}-\varphi_{i+1}^{n,L}\right)\right)\right]\\
  + u_{i-1}^{f,n+1}  \left [ \displaystyle \frac{\gamma \Delta t}{ \Delta x} \frac{c}{2}\left(-1-\varphi_i^{n,L}+\frac{1}{2}\left(\varphi_{i-1}^{n,R}-\varphi_{i-1}^{n,L}\right)\right)+\frac{\gamma \Delta t}{ \Delta x}\frac{k}{2}\left(-1+\frac{1}{2}\left(\varphi_{i-1}^{n,R}-\varphi_{i-1}^{n,L}\right)\right)\right]\\
  + u_{i+2}^{f,n+1}  \left [ \displaystyle \frac{\gamma \Delta t}{ \Delta x} \varphi_{i+1}^{n,R} \frac{k-c}{4}\right]
  + u_{i-2}^{f,n+1}  \left [ \displaystyle \frac{\gamma \Delta t}{ \Delta x} \varphi_{i-1}^{n,L} \frac{k+c}{4}\right]\\
 =
- \displaystyle \frac{\gamma \Delta t}{ \Delta x} \left(F_{i+1/2}^{n}- F_{i-1/2}^{n} \right)  + \displaystyle \frac{\gamma \Delta t}{ \Delta x} \left( f\left(u_i^{e,n}(x_{i+1/2})\right)-f\left(u_i^{e,n}(x_{i-1/2})\right)\right)+\frac{1-\gamma}{\gamma}u_i^{f,1}. 
\end{array}
\end{equation*}

\subsubsection{Test 1: stationary solution}

Let us consider the space interval  $[0, 2]$ and the time interval  $[0,1]$, $\alpha=1$ and $c=1$. The CFL parameter is set to 2. In order to check the well-balanced property, we consider the stationary solution of \eqref{eq_transport}
$$
u_0(x)=e^{x}$$ 
as initial condition. $L^1$-errors between the initial and final cell-averages have been computed for IEWBM$p$, $p=1,2$, using a 200-cell mesh (see Table \ref{transport_wbcheck}).
\begin{table}[ht]
\centering
\begin{tabular}{|c|cc|} \hline
IEWBM1 & \multicolumn{2}{c|}{IEWBM2}  \\
 &  PWCR & PWLR \\ \hline
1.63e-13 & 1.64e-13 & 1.57e-13 \\\hline
\end{tabular}
 \caption{Transport equation: Test 1. $L^1$-errors at $t = 1$ for IEWBM1 and IEWBM2 with piecewise contant (PWCR) or piecewise linear (PWL) reconstruction $\widetilde{Q}_i$.} 
\label{transport_wbcheck}
\end{table}

\subsubsection{Test 2:  perturbation of a stationary solution}
We consider now an initial condition that represents a perturbation of the stationary solution of the previous test case:
$$
u_0(x)=e^{x}+\frac{1}{2}e^{-100(x-0.3)^2}.$$ 
A reference solution has been obtained using IEWBM2 with piecewise linear reconstruction $\widetilde{Q}_i$ in a 6400-cell mesh. Table \ref{transport_order} shows the errors in $L^1$-norm with respect to the reference solution for EWBM1 and EWBM2 for both the piecewise constant and linear reconstructions. Notice that the errors decrease with the number of cells at the expected rate.

\begin{table}[ht]
\centering
\begin{tabular}{|c|cc|cc|cc|} \hline
Cells & \multicolumn{2}{c|}{IEWBM1}& \multicolumn{4}{c|}{IEWBM2}\\ \hline
 & & & \multicolumn{2}{c|}{PWCR} &\multicolumn{2}{c|}{PWLR}\\

  & Error&Order & Error&Order &Error& Order \\\hline
  
25& 7.27e-02 & - & 3.65e-1 & - & 1.99e-1 & - \\
50& 6.37e-02 & 0.19& 2.72e-01 & 0.42 &  1.09e-01  & 0.87 \\
100& 3.83e-02 & 0.73 & 1.57e-01 & 0.79 &  3.81e-02 & 1.51 \\
200&   2.17e-02 & 0.82 & 5.40e-02 & 1.52 &   9.39e-03 & 2.02 \\
400&   1.57e-02 & 0.47 & 1.45e-02 & 1.89 &    2.19e-03 & 2.10 \\
800&  6.62e-03 & 1.24& 3.70e-03 & 1.98 &    5.21e-04 & 2.06 \\
1600& 3.43e-03 & 0.95 & 9.24e-04 & 2.00 &    1.23e-04 & 2.08 \\
 \hline
\end{tabular}
\caption{ Transport equation: Test 2. Differences in $L^1$-norm between the reference  and the numerical solutions and convergence rates at $t=1$ for IEWBM1 and IEWBM2 with piecewise contant (PWCR) or piecewise linear (PWL) reconstruction $\widetilde{Q}_i$ with CFL=2.} 
\label{transport_order}
\end{table}

We have also compared the numerical solutions when different values of the CFL number are chosen (see Figures \ref{transport_difCFL_O1}-\ref{transport_difCFL_O2}).
\begin{figure}[ht]
 \begin{center}
   \includegraphics[width=0.8\textwidth]{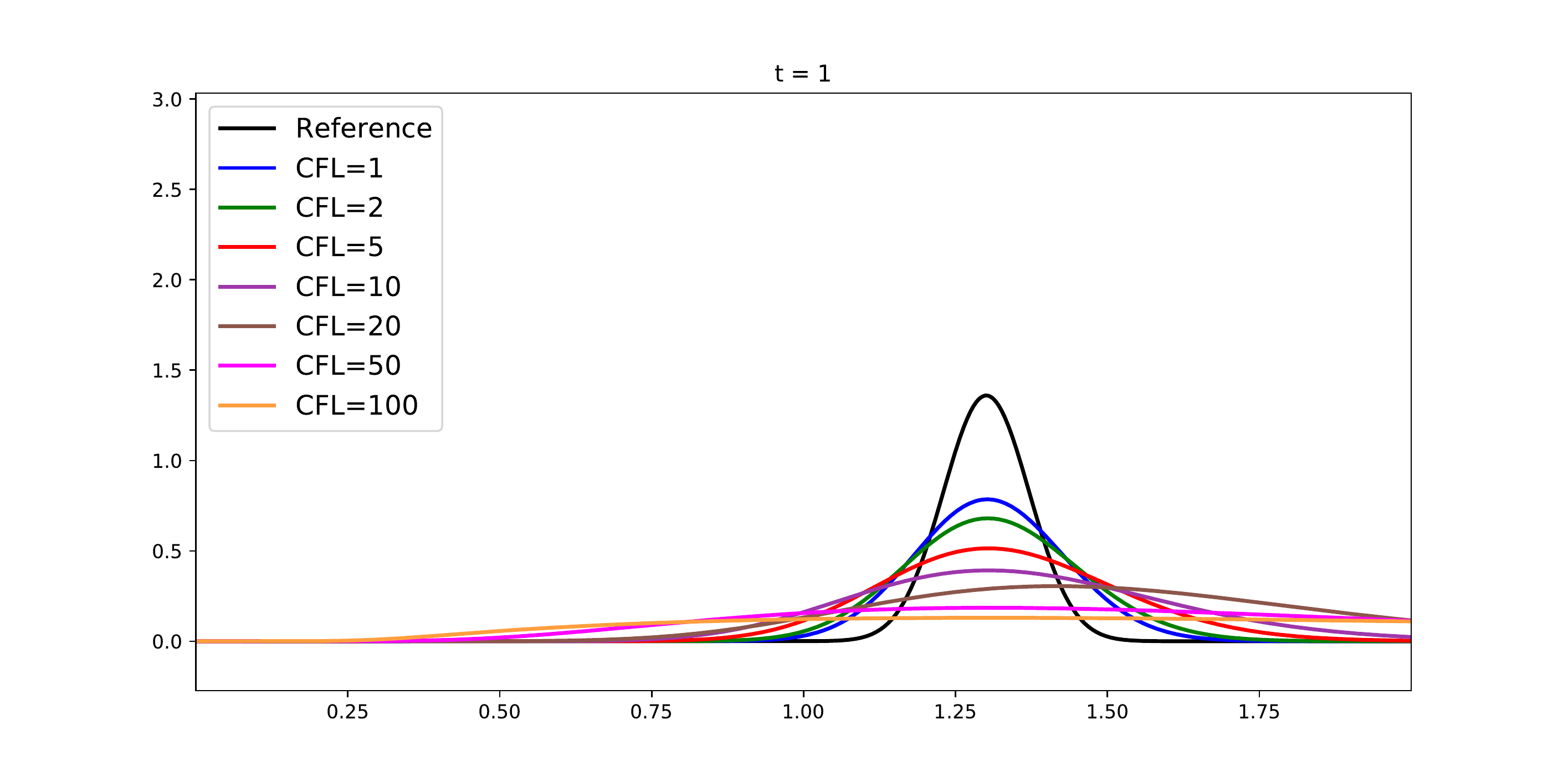}
     \caption{Transport equation: Test 2. Differences between $e^x$ and the  numerical solutions obtained with IEWBM1 using different CFL values  at $t=1$.} 
     \label{transport_difCFL_O1}
  \end{center}
 \end{figure}
 
 \begin{figure}[ht]
 \begin{center}
   \includegraphics[width=0.65\textwidth]{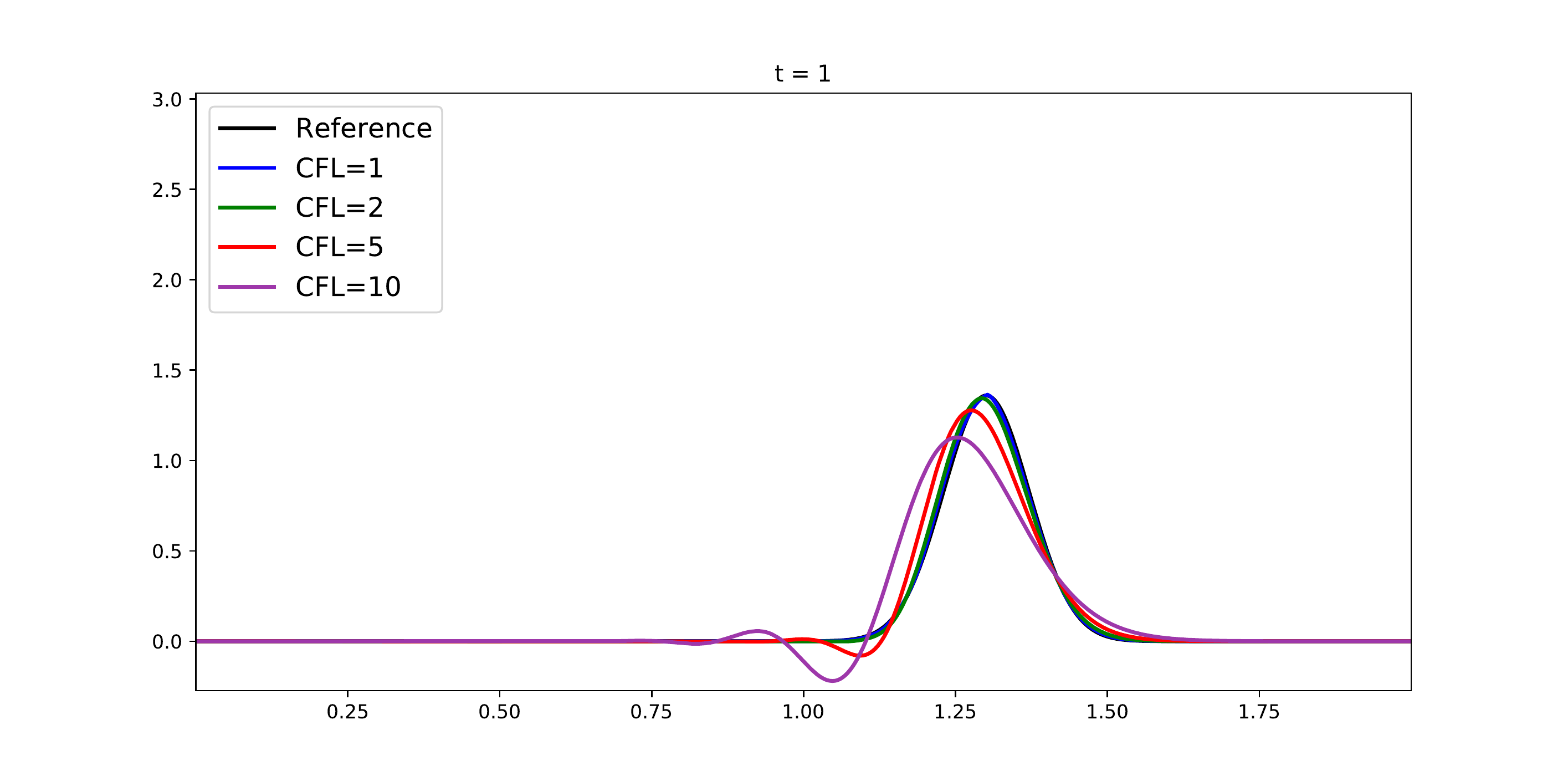}
   \includegraphics[width=0.65\textwidth]{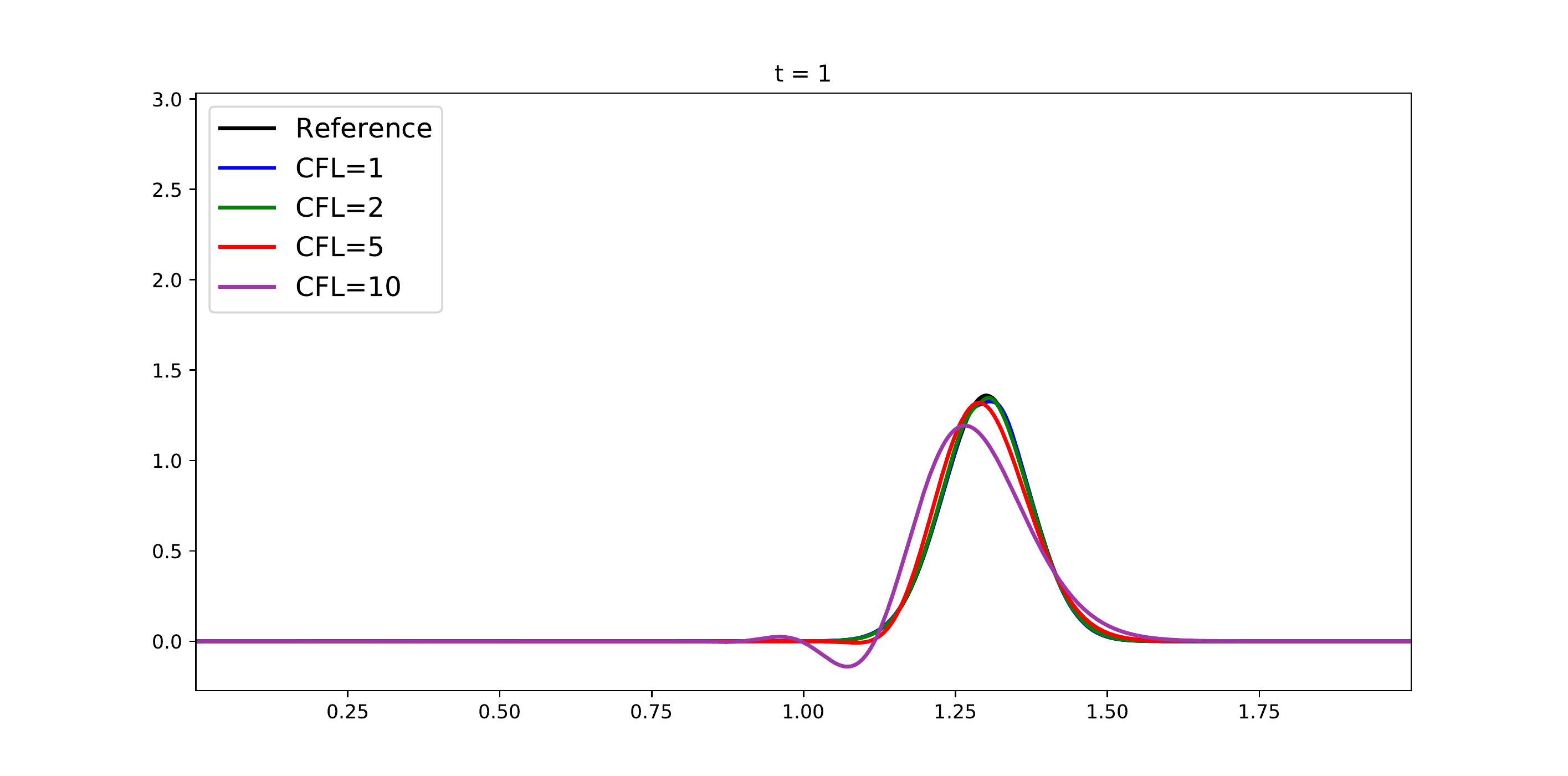}
     \caption{Transport equation: Test 2. Differences between $e^x$  and the numerical solutions obtained with IEWBM2 for different CFL values at $t=1$. Top: PWCR. Bottom:  PWLR. 
      }
     \label{transport_difCFL_O2}
  \end{center}
 \end{figure}
 
 Notice that, whereas for the first-order methods big values of CFL can be considered, the second order schemes present some oscillations related to the time integrator: a linear analysis performed by M. L\'opez-Fern\'andez shows that the method is stable for CFL values lower than $1 + \sqrt{2}$. 
 Moreover, the numerical solutions obtained with the piecewise constant reconstruction $\widetilde{Q}_i$ presents more oscillations. 
 
 Furthermore, due to the well-balanced property of the methods, they are able to recover the stationary solution once the perturbation has left the domain. This is shown in Table \ref{transport_wbpert}, where $L^1$-errors between the stationary solution and the numerical solutions at time $t = 5$  have been computed for IEWBM$i$, $i=1,2$, using a 400-cell mesh and CFL$=2$.

\begin{table}[ht]
\centering
\begin{tabular}{|c|cc|} \hline
IEWBM1 & \multicolumn{2}{c|}{IEWBM2}  \\
 &   PWCR & PWLR \\ \hline
4.15e-13 & 4.10e-13 & 4.09e-13 \\\hline
\end{tabular}
 \caption{Transport equation: Test 2. Differences in $L^1$-norm between the stationary and the numerical solution  at $t=5$ for IEWBM1 and IEWBM2 with piecewise contant (PWCR) or piecewise linear (PWL) reconstruction $\widetilde{Q}_i$.} 
\label{transport_wbpert}
\end{table}

\subsection{Burgers equation}
 Let us consider now the Burgers equation with a nonlinear source term
\begin{equation}\label{eq_burgers}
  u_t+\left(\frac{1}{2}u^2\right)_x= \alpha u^2,
\end{equation}
where  $\alpha \in \mathbb{R}$. Again, the stationary solutions can be easily obtained:
$$
u^e(x)=C e^{{\alpha} x}
$$
Then, given a family of cell values $\{u_i^n\}$, the stationary solution $u^{e,n}_i(x)$  sought in the first step of the reconstruction procedure is
$$
u^{e,n}_i(x)=u_i^n e^{\alpha (x-x_i)}.
$$
Exactly well-balanced methods are considered again. 

Due to the non-linearity of the flux and the source term, nonlinear systems have to be solved now to find the time fluctuations. A numerical method for nonlinear systems of algebraic equations is required. In particular, in this problem the Newton's method is considered. For the first-order methods, only one iteration of the Newton's method is performed. The Rusanov numerical flux 
 $$
 \mathbb{F}(u,v)=\frac{c}{2}(f(u)+f(v))-\frac{k}{2}(v-u),
 $$
 is considered again, where $k$ is the viscosity constant.
 
 \subsubsection{Test 1: stationary solution}
 We consider the space interval  $[0, 2]$, the time interval  $[0,1]$, $\alpha=1$, and CFL=2. As initial condition, we consider the stationary solution $$
u_0(x)=e^{x}.$$
$L^1$ errors between the initial and final cell-averages have been computed for IEWBM$p$, $p=1,2$, using a 200-cell mesh (see Table \ref{burgers_wbcheck}).
\begin{table}[ht]
\centering
\begin{tabular}{|c|cc|} \hline
IEWBM1 & \multicolumn{2}{c|}{IEWBM2}  \\
 &   PWCR  & PWLR  \\ \hline
1.54e-13 & 1.32e-13 & 1.56e-13 \\\hline
\end{tabular}
 \caption{Burgers equation: Test 1. $L^1$-errors between the stationary and the numerical solution at $t = 1$ for IEWBM1 and IEWBM2 with piecewise contant (PWCR) or piecewise linear (PWLR) reconstruction $\widetilde{Q}_i$.
 } 
\label{burgers_wbcheck}
\end{table}

\subsubsection{Test 2: perturbation of a stationary solution}
We consider now the initial condition:
$$
u_0(x)=e^{x}+0.4e^{-25(x-0.4)^2}.$$
A reference solution has been computed using again IEWBM2 with piecewise linear reconstruction $\widetilde{Q}_i$ with a 3200-cell mesh. The numerical solutions obtained using  different values of the CFL number are shown in Figures \ref{burgers_difCFL_O1}-\ref{burgers_difCFL_O2}. Notice that, unlike the linear case, no oscillations appear for second-order methods for large CFL values. Moreover, similar results are obtained with the piecewise constant and linear reconstructions $\widetilde{Q}_i$.

\begin{figure}[ht]
 \begin{center}
   \includegraphics[width=0.8\textwidth]{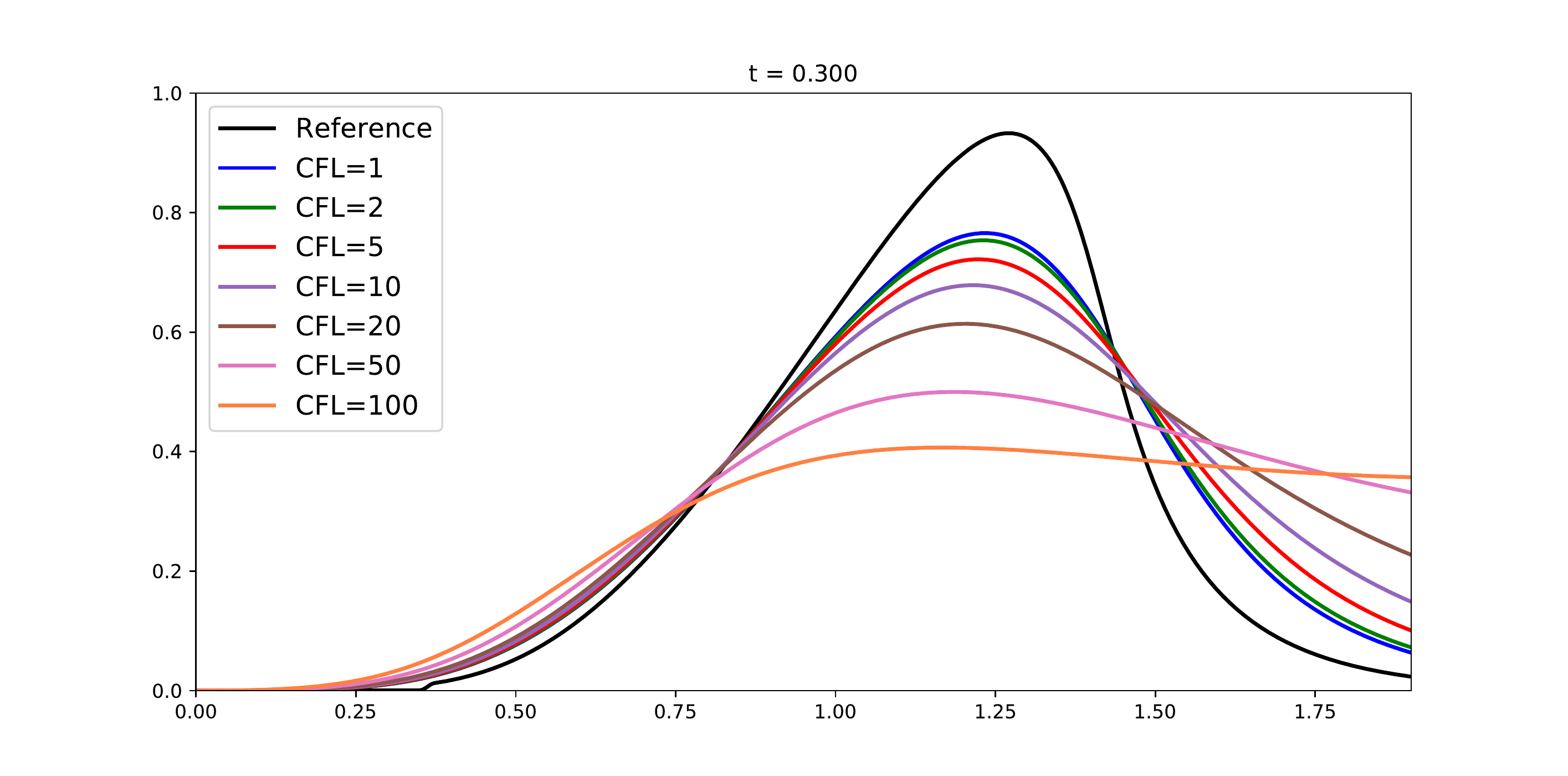}
     \caption{Burgers equation: Test 2. Differences between $e^x$ and the  numerical solutions obtained with IEWBM1 using different CFL values at $t=0.3$.} 
     \label{burgers_difCFL_O1}
  \end{center}
 \end{figure}
 
 \begin{figure}[ht]
 \begin{center}
   \includegraphics[width=0.65\textwidth]{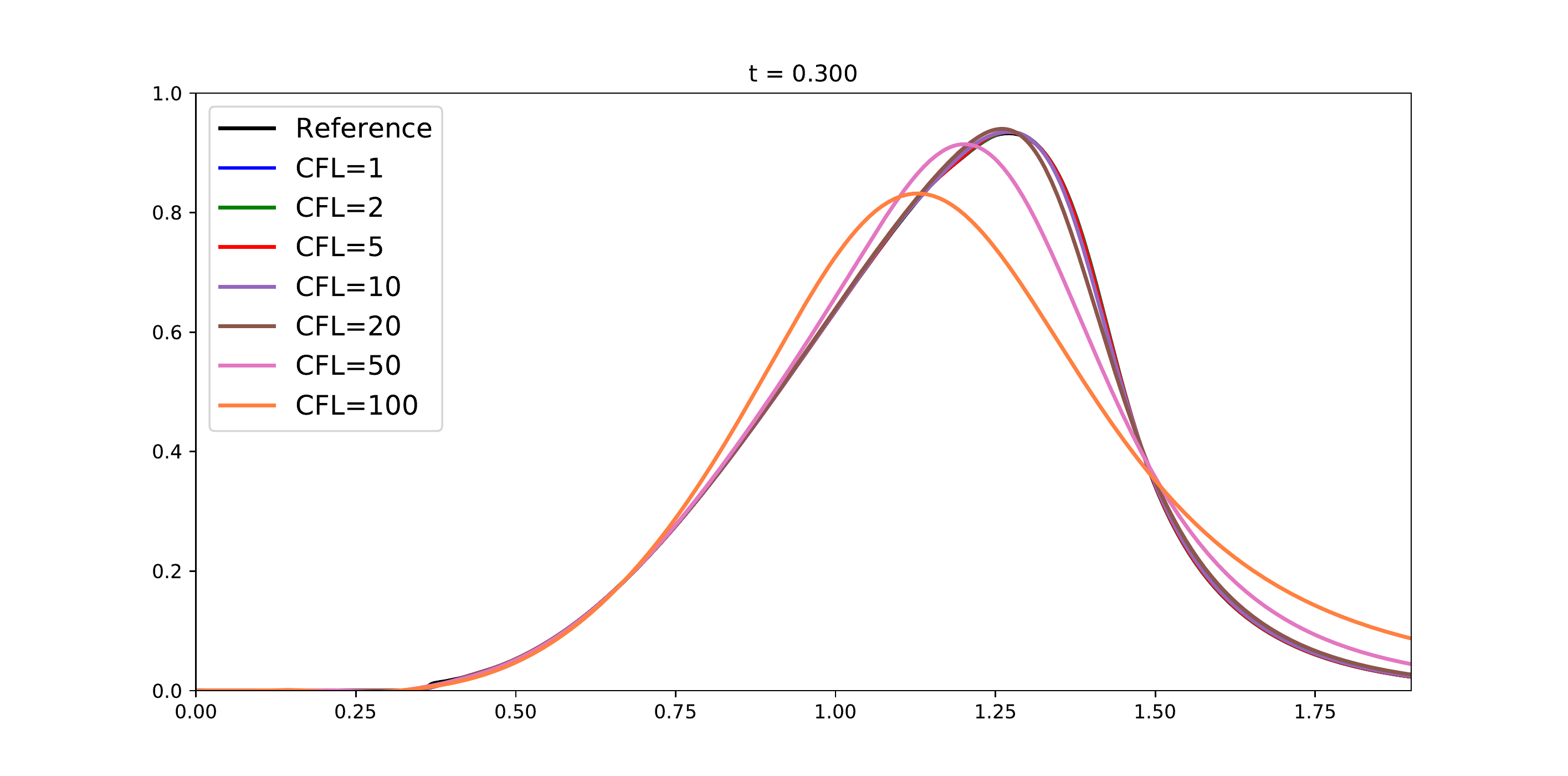}
   \includegraphics[width=0.65\textwidth]{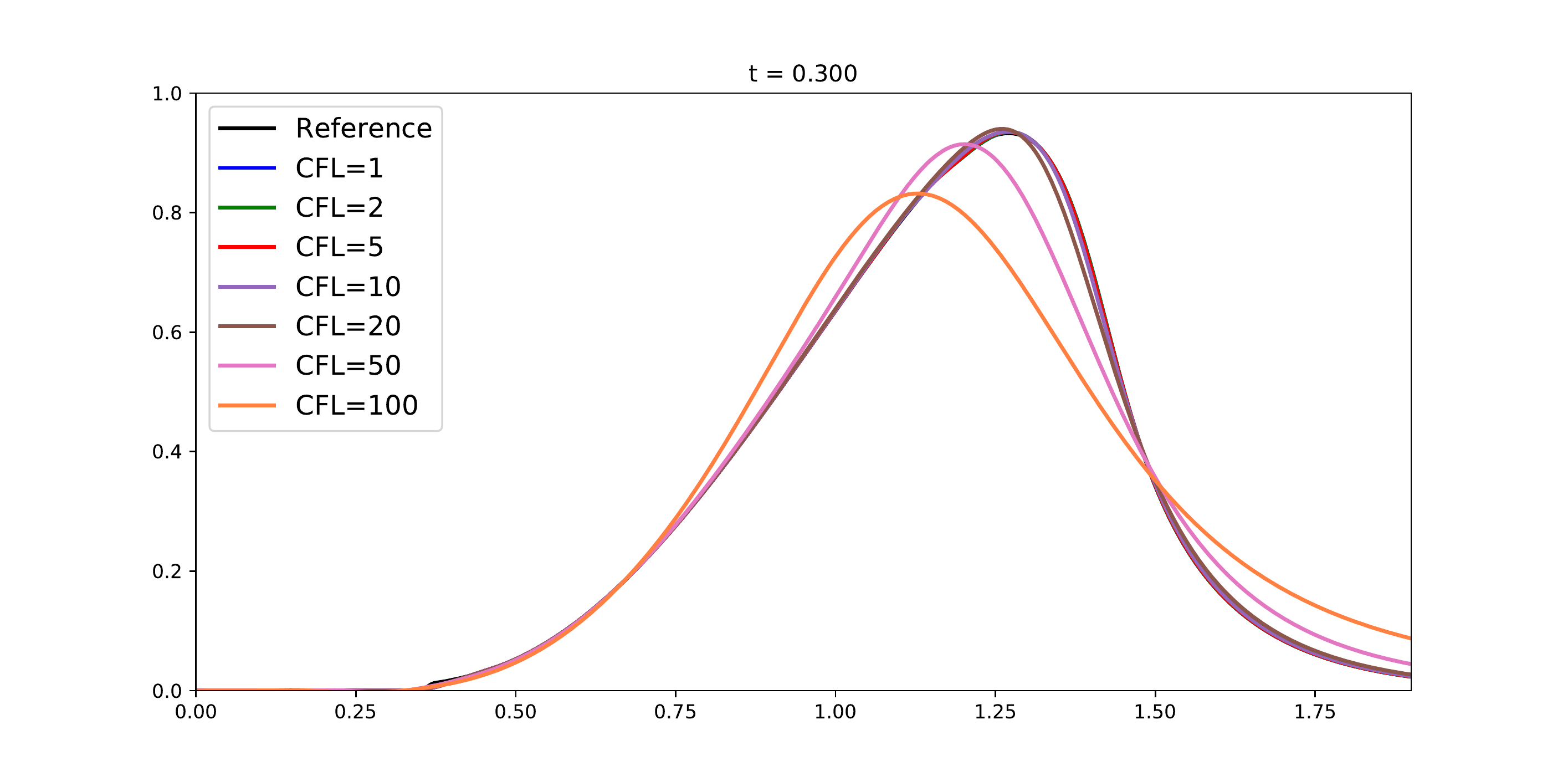}
     \caption{Burgers equation: Test 2. Differences between $e^x$ and the  numerical solutions obtained with IEWBM2 using different CFL values at $t=0.3$. Top: PWCR. Bottom:  PWLR.} 
      \label{burgers_difCFL_O2}
  \end{center}
 \end{figure}
 
Table \ref{burgers_wbpert} shows the $L^1$-differences between the underlying stationary solution and the numerical solutions obtained with IEWBM$p$, $p=1,2$ at time $t = 10$ (once the perturbation has left the computational domain) using a 400-cell mesh and CFL$=2$. Again, the stationary solution is recovered to machine precision.

\begin{table}[ht]
\centering
\begin{tabular}{|c|cc|} \hline
IEWBM1 & \multicolumn{2}{c|}{IEWBM2}  \\
 &   PWCR & PWLR \\ \hline
1.54e-13 & 1.32e-13 & 1.56e-13 \\\hline
\end{tabular}
 \caption{Burgers equation: Test 2. Differences in $L^1$-norm between the stationary and the numerical solution  at $t=10$ for IEWBM1 and IEWBM2 with piecewise contant (PWCR) or piecewise linear (PWL) reconstruction $\widetilde{Q}_i$.} 
\label{burgers_wbpert}
\end{table}

\subsubsection{Test 3: order test}
 
We consider now the initial condition:
$$u_0(x)=0.1e^{x}+0.5e^{-25(x-1.0)^2}.$$
A reference solution has been computed with IEWBM2 and the piecewise linear reconstruction using a 6400-cell mesh. Table \ref{burgers_order} shows the errors in $L^1$-norm with respect to the reference solution for EWBM1 and EWBM2 for both the piecewise constant and linear reconstructions: as it can be checked, the errors decrease with the number of cells at the expected convergence rate.

\begin{table}[ht]
\centering
\begin{tabular}{|c|cc|cc|cc|} \hline
Cells & \multicolumn{2}{c|}{IEWBM1}& \multicolumn{4}{c|}{IEWBM2}\\ \hline
 & & & \multicolumn{2}{c|}{PWCR} &\multicolumn{2}{c|}{PWLR}\\

  & Error&Order & Error&Order &Error& Order \\\hline
  
25& 1.67e-02 & - & 1.81e-03 & - & 1.81e-03 & -  \\
50& 7.42e-02 & 1.17& 4.76e-04 & 1.93 &  4.76e-04  & 1.93 \\
100& 4.79e-03 & 0.63 & 1.06e-04 & 2.16 &  1.06e-04 & 2.16 \\
200&   2.74e-02 & 0.81 & 2.99e-05 & 1.83 &   2.99e-05 & 1.83 \\
400&   1.48e-03 & 0.89 & 8.00e-06 & 1.90 &    8.00e-03 & 1.90 \\
800&  7.69e-04 & 0.95& 2.02e-06 & 1.99 &    2.02e-06 & 1.99 \\
1600& 3.93e-04 & 0.97 & 4.85e-07 & 2.00 &    4.85e-07 & 2.06 \\
 \hline
\end{tabular}
\caption{Burgers equation: Test 3. Differences in $L^1$-norm between the reference and the numerical solutions  at $t=0.5$ for IEWBM1 and IEWBM2 with piecewise contant (PWCR) or piecewise linear (PWLR) reconstruction $\widetilde{Q}_i$.} 
\label{burgers_order}
\end{table}

\subsection{Shallow water equations}
Let us consider now the 1d shallow water system with Manning friction
$$
U_t+f(U)_x=S_1(U)H_{x}+S_2(U),
$$
where
$$
U=\begin{pmatrix}
h \\
q \\
\end{pmatrix}, \quad f(U)=\begin{pmatrix}
q \\
\displaystyle \frac{q^2}{h}+\displaystyle \frac{g}{2}h^2\\
\end{pmatrix}, \quad S_1(U)=\begin{pmatrix}
0\\
gh\\
\end{pmatrix}, \quad S_2(U)=\begin{pmatrix}
0\\
 - \displaystyle \frac{k q |q|}{h^{\mu}}\\
\end{pmatrix}.
$$
The variable $x$ makes reference to the axis of the channel and $t$ is the time; $q(x,t)$ and $h(x,t)$ are the discharge and the thickness, respectively; $g$ is the gravity; $H(x)$ is the depth function measured from a fixed reference level; $k$ is the Manning friction coefficient and $\mu$ is set to  $ \frac{7}{3}$.

In order to make the numerical experiments more realistic, we adopt dimensional units, in which time is measured in seconds, and space is measured in meters.

The eigenvalues of the system are 
$$\lambda^\pm=u\pm c, $$
with $c = \sqrt{gh}$. The flow regime is characterized by the Froude number: 
\begin{equation}
Fr(U)= \displaystyle \frac{|u|}{c}.
\end{equation}
The flow is subcritical if $Fr<1$, critical if $Fr=1$ or supercritical if $Fr>1$.

The stationary solutions satisfy the ODE system 
\begin{equation}\label{u'=Gswf}
\begin{cases}
\left( - {u^2}+ gh \right) h_x= ghH_{x}- \displaystyle \frac{k q |q|}{h^{\mu}} ,\\
q_x=0.
\end{cases}
\end{equation}
Since the explicit expression of the stationary solutions is not available, well-balanced methods will be designed in which RK-collocation methods are used to approximate them. Different schemes are considered:
    \begin{itemize}
        \item Fully implicit schemes, where the flux and the source term are treated implicitly.
        \item Semi-implicit schemes for problems without  friction, in which the advection term $\left(\displaystyle \frac{q^2}{h}\right)$ is treated explicitly and the equation of $h$, the pressure term $\left(\displaystyle\frac{1}{2} g h^2\right)$ and the source term are treated implicitly.
        \item Semi-implicit schemes for problems with friction, in which only the source term $\left( - \displaystyle \frac{k q |q|}{h^{\mu}}\right)$ is treated implicitly.
    \end{itemize}
  While the implementation of the two first types of methods above lead to solve coupled nonlinear algebraic systems at every stage of the RK solvers, in the case of the third types, local nonlinear equations have to be solved at every cell. Fixed-point iterations are considered in all cases.
 \subsubsection{Test 1: stationary solution}
Let us first check the well-balanced property of the methods for the model without friction ($k=0$). We consider $x \in [0,3]$, $t \in [0,1]$ and CFL$=2.0$. As initial condition, we consider the subcritical stationary solution which solves the Cauchy problem: 
\begin{equation*}
\begin{cases}
q_x=0,\\
(\displaystyle - u^2 + gh) h_x= ghH_x,\\
h(0)=2.0+H(0), \, q(0)=3.5,\\
\end{cases}
\end{equation*}
where the depth function is given by
\begin{equation}\label{sw_fondo_subcrit}
H(x)= \left\{ \begin{array}{lcc}
             -0.25(1+\cos(5 \pi (x+0.5))) &   \text{if}  & 1.3 \leq x \leq 1.7, \\
             \\ \hspace{20mm} 0 &  &\text{otherwise}, 
             
             \end{array}
   \right.
\end{equation}
(see Figure \ref{sw_wbcheck_cini}). Table \ref{sw_wbcheck_errors} shows the $L^1$ errors between the initial and final cell-averages for IWBM$p$, SIWBM$p$, $p=1,2$, using a 200-cell mesh.

 \begin{figure}[ht]
 \begin{center}
   \includegraphics[width=0.8\textwidth]{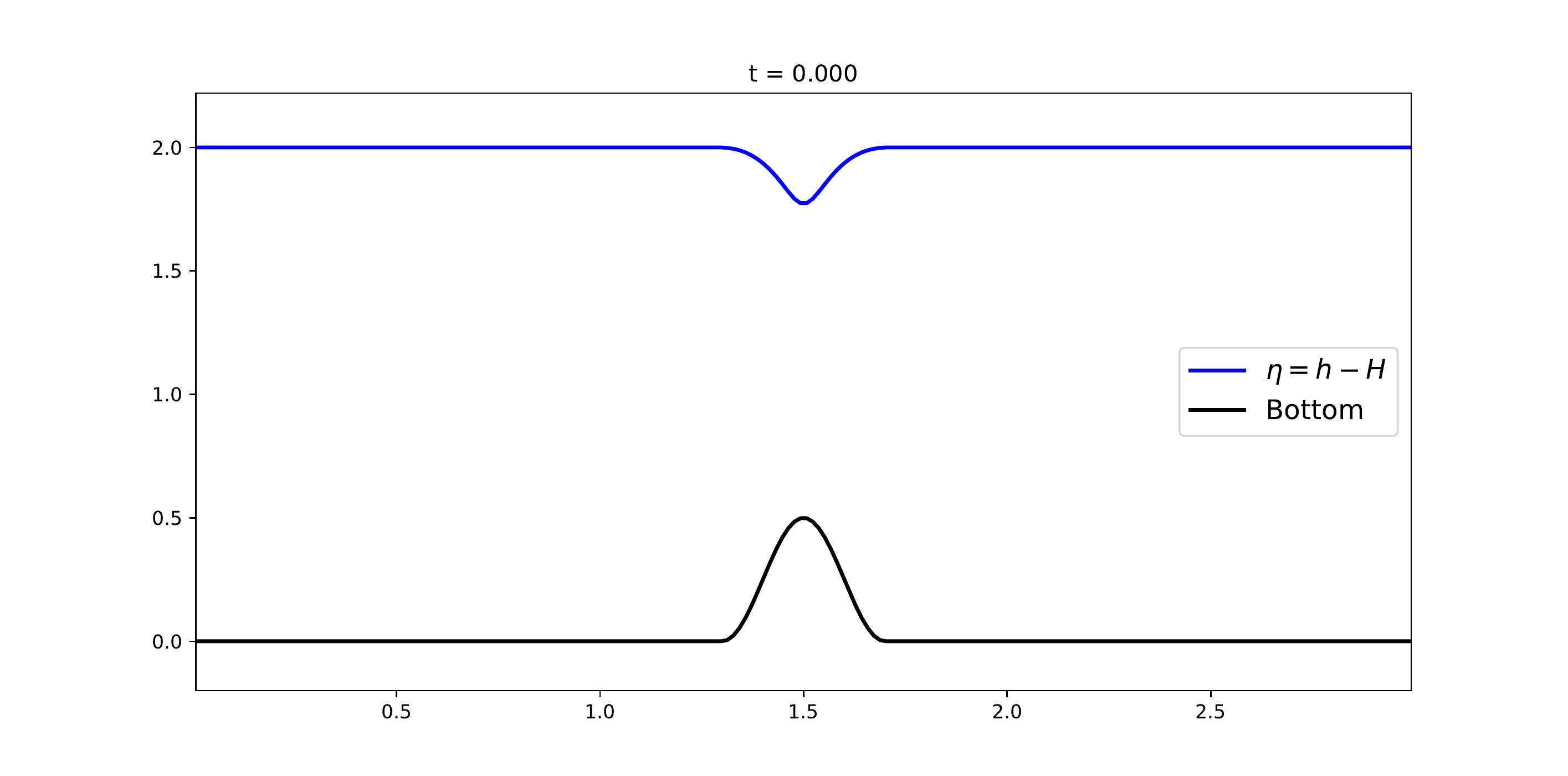}
     \caption{Shallow water equations without friction: Test 1. Initial condition: subcritical stationary solution.} 
     \label{sw_wbcheck_cini}
  \end{center}
 \end{figure}
    
  \begin{table}[ht]
\centering
\begin{tabular}{|cc|cccc|} \hline\multicolumn{6}{|c|}{\textbf{Implicit methods}}\\\hline
\multicolumn{2}{|c|}{IWBM1} & \multicolumn{4}{c|}{IWBM2}  \\
\multicolumn{2}{|c|} {  }&   \multicolumn{2}{c}{PWCR} & \multicolumn{2}{c|}{PWLR} \\ \hline
h& q & h & q &h & q \\\hline
5.33e-15 & 4.88e-15 & 3.55e-15  & 7.55e-15 & 3.55e-15 & 6.22e-15 \\\hline
 \multicolumn{6}{c}{}\\
 \multicolumn{6}{c}{}\\\hline
\multicolumn{6}{|c|}{\textbf{Semi-implicit methods}}\\\hline
\multicolumn{2}{|c|}{SIWBM1} & \multicolumn{4}{c|}{SIWBM2}  \\
\multicolumn{2}{|c|} {  }&   \multicolumn{2}{c}{PWCR} & \multicolumn{2}{c|}{PWLR} \\ \hline
h& q & h & q &h & q \\\hline
2.00e-15 & 7.11e-15 & 5.11e-15  & 8.00e-15 & 5.55e-15 & 8.44e-15 \\\hline
\end{tabular}
 \caption{ Shallow water equations without friction: Test 1. Differences in $L^1$-norm between the stationary and the numerical solution  at $t=1$ for IWBM1, SIWBM1, IWBM2 and SIWBM2 with piecewise contant (PWCR) or piecewise linear (PWLR) reconstruction $\widetilde{Q}_i$ for a 200-cell mesh.} 
\label{sw_wbcheck_errors}
\end{table}  
  
 \subsubsection{Test 2: order test}  
We now simulate a perturbation of a non-stationary smooth solution for the model without friction ($k=0$). In particular, we consider $x \in [-5,5]$, $t \in [0,2]$ CFL$=2.0$, and the depth function:
\begin{equation}\label{sw_fondogaussiana}
H(x)= 1.0-0.5e^{-x^2}.
\end{equation}
The initial condition is given by:
$$
q_0(x)=0, \quad h_0(x)= 0.1e^{-5.0x^2},
$$
(see Figure \ref{sw_gauss_cini}). A 200-cell mesh is considered and a reference solution with a 1600-cell mesh using IWBM2 with the piecewise linear reconstruction has been obtained. The different implicit and semi-implicit methods have been compared (see Figure \ref{sw_gauss_todos}). 
 As expected, the numerical solutions obtained with first-order schemes are more diffusive than those given by second-order methods. Moreover, semi-implicit methods give better results than fully implicit schemes in the first-order case. Concerning the second-order schemes, the piecewise constant and linear reconstructions give similar results. No spurious oscillations appear for CFL$=2$.

\begin{figure}[ht]
 \begin{center}
   \includegraphics[width=0.8\textwidth]{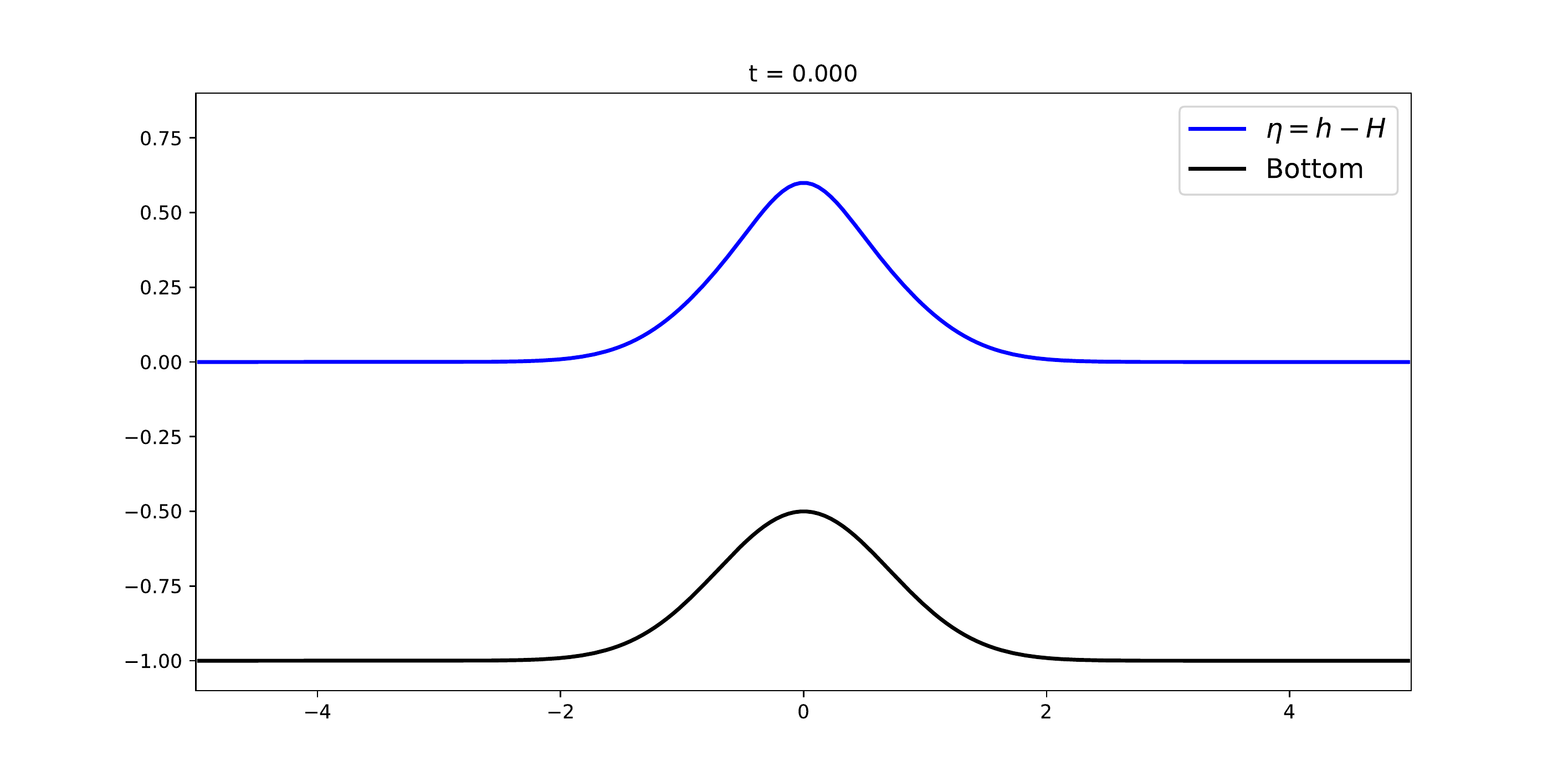}
     \caption{Shallow water equations without friction: Test 2. Initial condition.
      }
     \label{sw_gauss_cini}
  \end{center}
 \end{figure}

  \begin{figure}[ht]
 \begin{center}
   \includegraphics[width=0.49\textwidth]{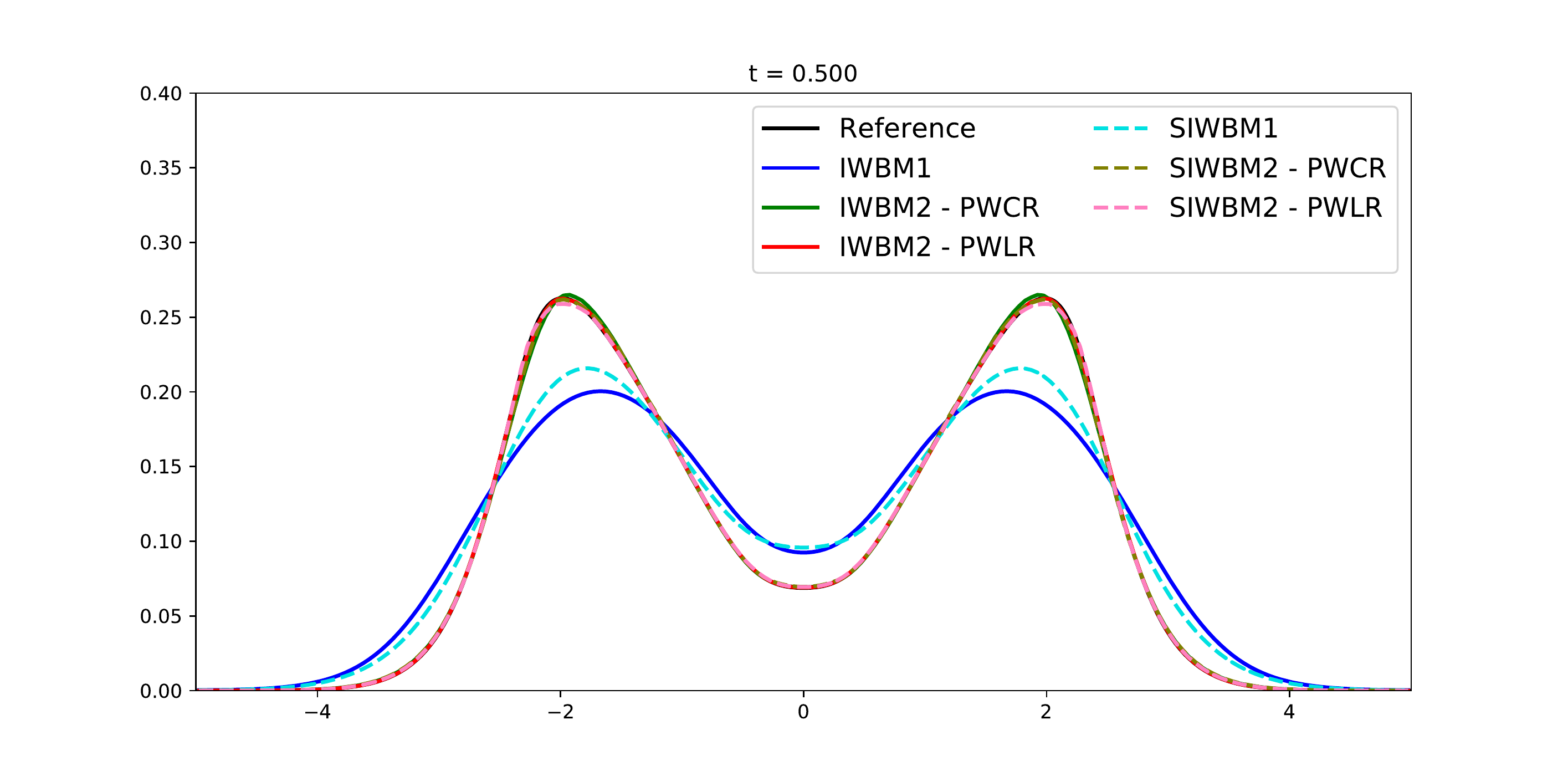}
   \includegraphics[width=0.49\textwidth]{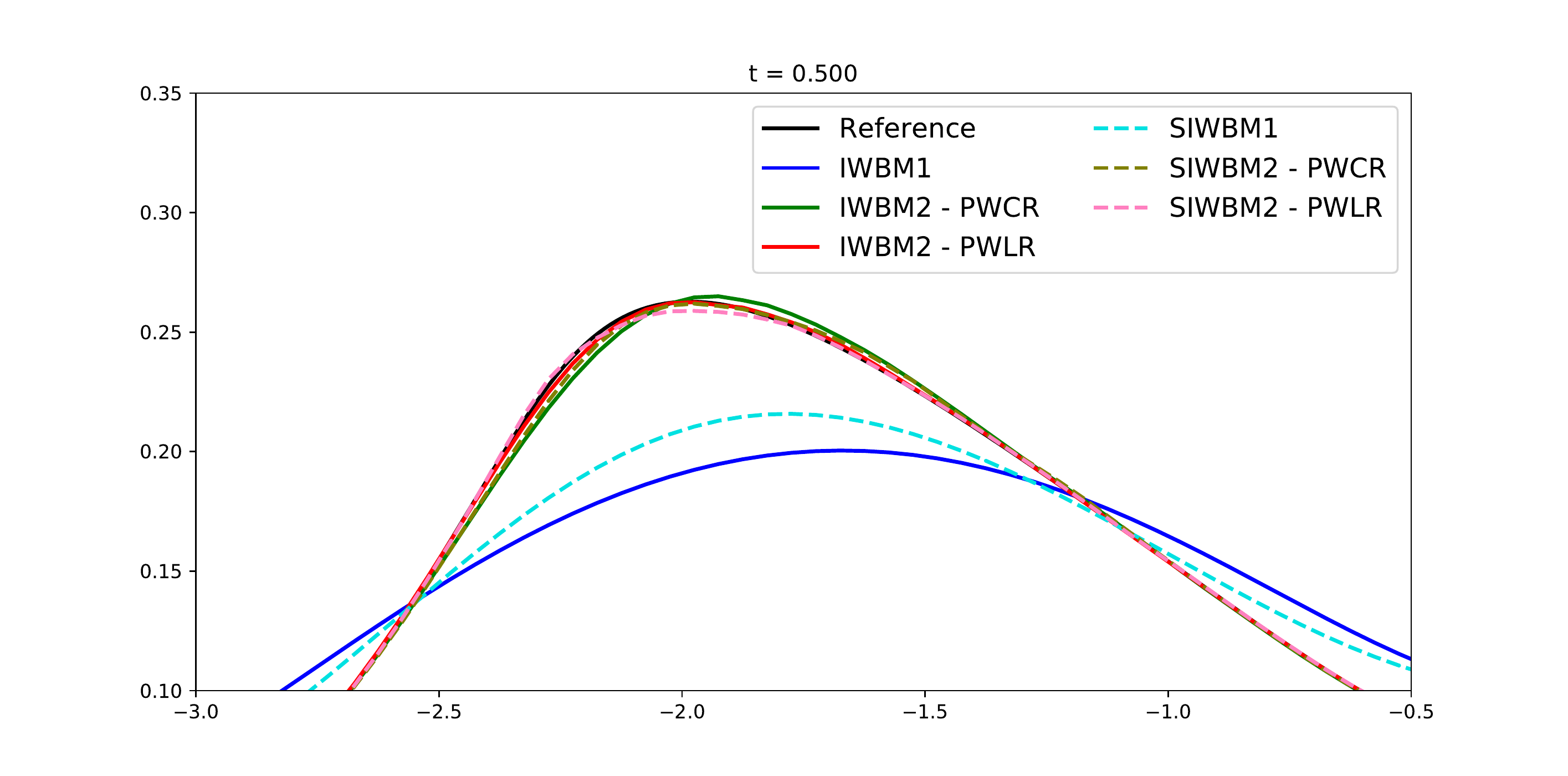}
   \includegraphics[width=0.49\textwidth]{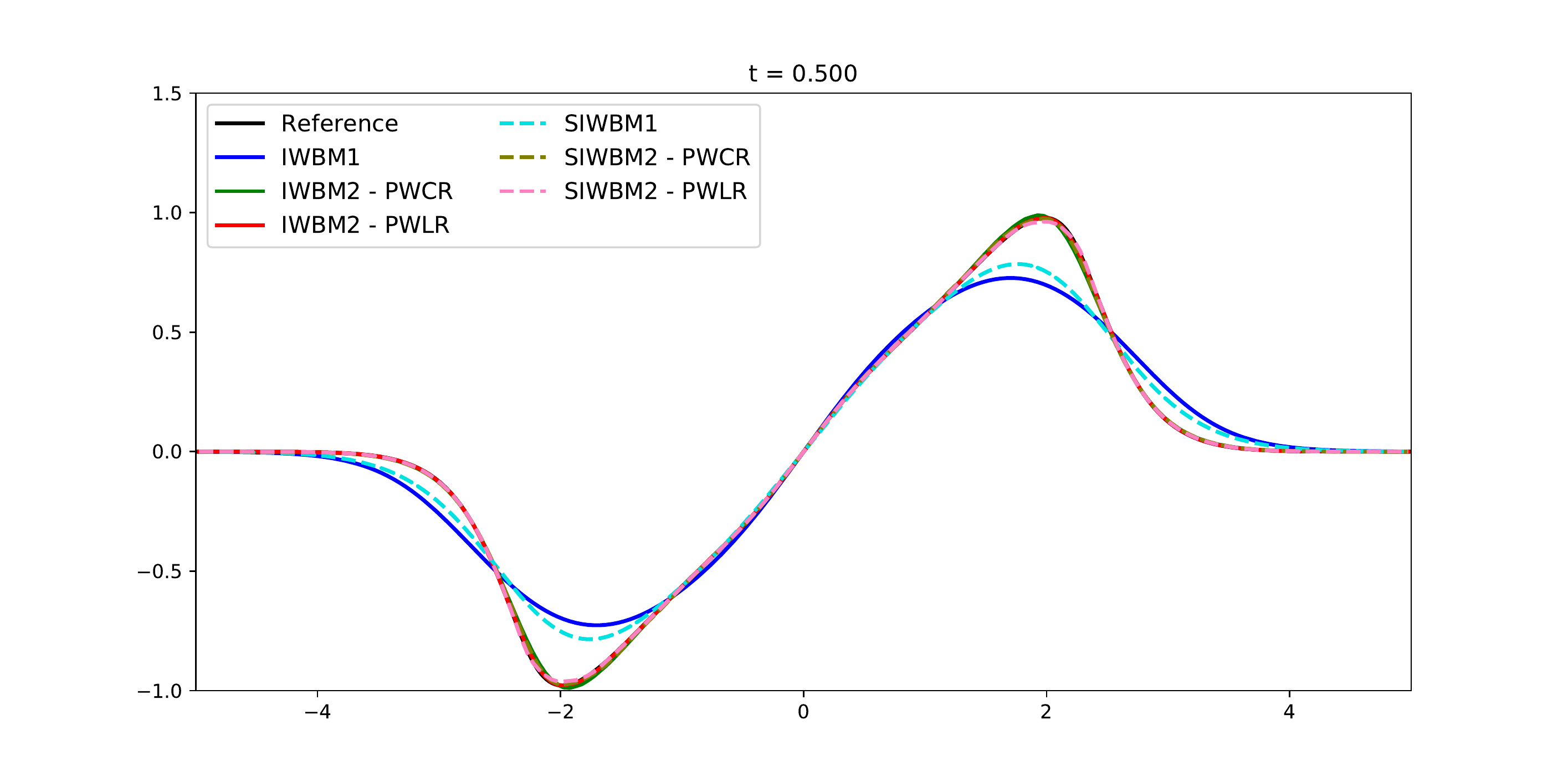}
   \includegraphics[width=0.49\textwidth]{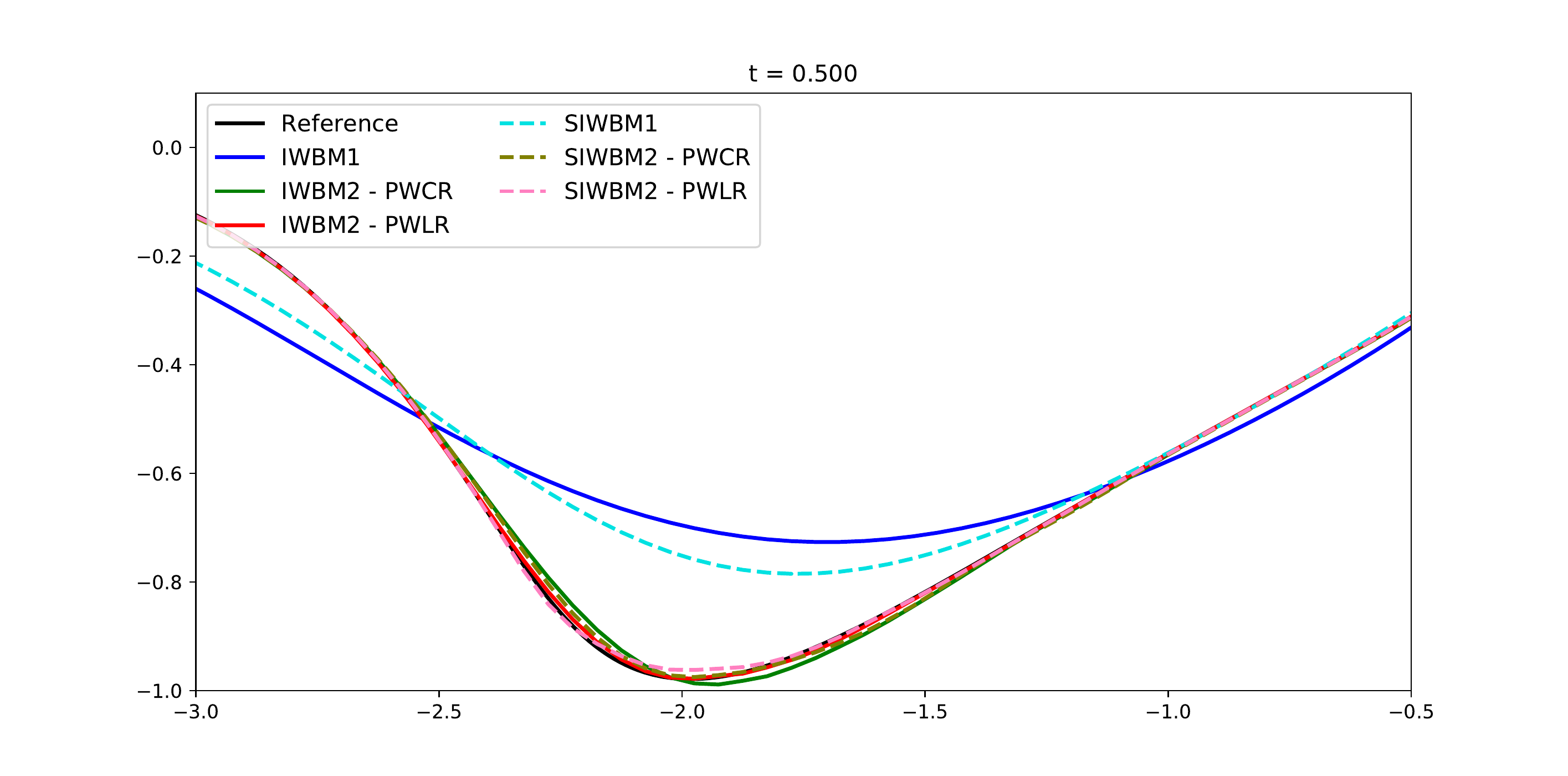}
     \caption{Shallow water equations without friction: Test 2. Numerical solutions for IWBM$p$, SIWBM$p$, $p=1,2$ with CFL$=2$ at $t=0.5$. Top: $\eta$. Bottom: $q$. } 
     \label{sw_gauss_todos}
  \end{center}
 \end{figure}
 
Tables \ref{sw_gauss_order_I_h}-\ref{sw_gauss_order_I_q}-\ref{sw_gauss_order_SI_h}-\ref{sw_gauss_order_SI_q} show the order for the implicit and semi-implicit methods. The expected convergence rates have been obtained for both variables $h$ and $q$.

\begin{table}[ht]
\centering
\begin{tabular}{|c|cc|cc|cc|} \hline
Cells & \multicolumn{2}{c|}{IWBM1}& \multicolumn{4}{c|}{IWBM2}\\ \hline
 & & & \multicolumn{2}{c|}{PWCR} &\multicolumn{2}{c|}{PWLR}\\

  & Error&Order & Error&Order &Error& Order \\\hline
  
25& 2.60e-01 & - & 2.90e-01 & - & 1.57e-01 & -  \\
50& 2.32e-01 & 0.16& 1.31e-01 & 1.14 &  4.91e-02  & 1.68 \\
100& 2.06e-01 & 0.17 & 4.90e-02 & 1.42 &  1.37e-02 & 1.84 \\
200&   9.88e-01 & 1.06 & 1.42e-01 & 1.78 &   3.52e-03 & 1.96 \\
400&   4.20e-02 & 1.23 & 3.73e-03 & 1.94 &    8.48e-04 & 2.05 \\
 \hline
\end{tabular}
\caption{Shallow water equations without friction: Test 2.  Differences in $L^1$-norm with respect to the reference solution and convergence rates for $h$ at $t=0.5$ for IWBM1 and IWBM2 with piecewise constant (PWCR) or picewise linear (PWLR) reconstruction $\widetilde Q_i$ for the thickness $h$.} 
\label{sw_gauss_order_I_h}
\end{table}

\begin{table}[ht]
\centering
\begin{tabular}{|c|cc|cc|cc|} \hline
Cells & \multicolumn{2}{c|}{IWBM1}& \multicolumn{4}{c|}{IWBM2}\\ \hline
 & & & \multicolumn{2}{c|}{PWCR} &\multicolumn{2}{c|}{PWLR}\\

  & Error&Order & Error&Order &Error& Order \\\hline
  
25& 1.35 & - & 1.17 & - & 5.55e-01 & -  \\
50& 1.04 & 0.37& 5.62e-01 & 1.05 &  2.04e-01  & 1.45 \\
100& 7.38-01 & 0.40 & 1.92e-01 & 1.55 &  5.56e-02 & 1.87 \\
200&   3.98e-01 & 0.89 & 5.72e-02 & 1.74 &   1.44e-02 & 1.95 \\
400&   1.95e-01 & 1.03 & 1.51e-02 & 1.92 &    3.48e-03 & 2.05 \\
 \hline
\end{tabular}
\caption{Shallow water equations without friction: Test 2.  Differences in $L^1$-norm with respect to the reference solution and convergence rates for $h$ at $t=0.5$ for IWBM1 and IWBM2 with piecewise constant (PWCR) or picewise linear (PWLR) reconstruction $\widetilde Q_i$ for the discharge $q$.} 
\label{sw_gauss_order_I_q}
\end{table}

\begin{table}[ht]
\centering
\begin{tabular}{|c|cc|cc|cc|} \hline
Cells & \multicolumn{2}{c|}{SIWBM1}& \multicolumn{4}{c|}{SIWBM2}\\ \hline
 & & & \multicolumn{2}{c|}{PWCR} &\multicolumn{2}{c|}{PWLR}\\

  & Error&Order & Error&Order &Error& Order \\\hline
  
25& 4.82e-01 & - & 1.41e-01 & - & 1.14e-01 & -  \\
50& 3.70e-01 & 0.38& 5.34e-02 & 1.40 &  2.86e-02  & 1.99 \\
100& 2.24e-01 & 0.72 & 1.72e-02 & 1.63 &  6.33e-03 & 2.18 \\
200&   1.39e-01 & 0.68 & 4.55e-03 & 1.92 &   1.53e-03 & 2.05 \\
400&   7.38e-02 & 0.92 & 1.16e-03 & 1.97 &    3.62e-04 & 2.08 \\
 \hline
\end{tabular}
\caption{Shallow water equations without friction: Test 2.  Differences in $L^1$-norm with respect to the reference solution and convergence rates for $h$ at $t=0.5$ for SIWBM1 and SIWBM2 with piecewise constant (PWCR) or picewise linear (PWLR) reconstruction $\widetilde Q_i$ for the thickness $h$.}  
\label{sw_gauss_order_SI_h}
\end{table}

\begin{table}[ht]
\centering
\begin{tabular}{|c|cc|cc|cc|} \hline
Cells & \multicolumn{2}{c|}{SIWBM1}& \multicolumn{4}{c|}{SIWBM2}\\ \hline
 & & & \multicolumn{2}{c|}{PWCR} &\multicolumn{2}{c|}{PWLR}\\

  & Error&Order & Error&Order &Error& Order \\\hline
  
25& 1.74 & - & 6.10e-01 & - & 3.33e-01 & -  \\
50& 1.47 & 0.24& 2.23e-01 & 1.45 &  9.40e-02  & 1.83 \\
100& 9.83e-01 & 0.74 & 6.88e-02 & 1.69 &  2.25e-02 & 2.06 \\
200&   5.83e-01 & 0.60 & 1.84e-02 & 1.91 &   5.64e-03 & 2.00 \\
400&   3.01e-01 & 0.95 & 4.69e-02 & 1.97 &    1.35e-03 & 2.07 \\
 \hline
\end{tabular}
\caption{Shallow water equations without friction: Test 2.  Differences in $L^1$-norm with respect to the reference solution and convergence rates for $h$ at $t=0.5$ for IWBM1 and IWBM2 with piecewise constant (PWCR) or picewise linear (PWLR) reconstruction  $\widetilde Q_i$ for the discharge $q$.
}  
\label{sw_gauss_order_SI_q}
\end{table}

As expected, second-order schemes which make use of piecewise linear reconstruction for $\tilde{Q}_i$ are systematically more accurate than the ones based on piecewise constant reconstruction, even if the order of accuracy is the same. It is interesting to observe that semi-implicit schemes are more accurate than fully implicit schemes and even of explicit schemes. This makes semi-implicit schemes the most cost-effective since they are more accurate and less expensive than fully implicit schemes, and they allow larger CFL numbers and are even more accurate than explicit schemes. 
Fully implicit schemes are more dissipative than semi-implicit schemes, which explains the higher accurate for the same grid and CFL number. Furthermore, for low Froude number flow, they are also less dissipative than explicit schemes. The reason is that explicit schemes have to use a larger numerical viscosity than semi-implicit ones in the Rusanov numerical flux function. This effect is  well known and discussed in the detail in the literature for numerical methods for all Mach-number flows (see for example 
\cite{dellacherie2010analysis} for the analysis of the phenomenon, \cite{boscarino2018all,
arun2012asymptotic, avgerinos2019linearly} for second order accurate finite volume schemes to treat all Mach number flows in compressible Euler equations, \cite{bosscarino2022lowmach} for high order schemes, and \cite{miczek2015new} for applications in astrophysics).

\subsubsection{Test 3: shock waves}  
We consider again the model without friction ($k=0$) with a discontinuous initial condition that generates two shock waves traveling in oposite directions.  We consider the space interval  $[-5, 5]$ and the time interval  $[0,0.5]$ and CFL=2. The depth function is again given by \eqref{sw_fondogaussiana}.
As initial condition, we consider the functions:
$$
q_0(x)=0, \quad h_0(x)=\left\{ \begin{array}{lc}
             H(x) &   \text{if  $|x|\geq 1$;} \\
             \\ H(x)+0.1 &  \text{if $ |x|<1 $;} \\
             \end{array}
   \right.
$$
(see Figure \ref{sw_shock}).

\begin{figure}[ht]
 \begin{center}
   \includegraphics[width=0.8\textwidth]{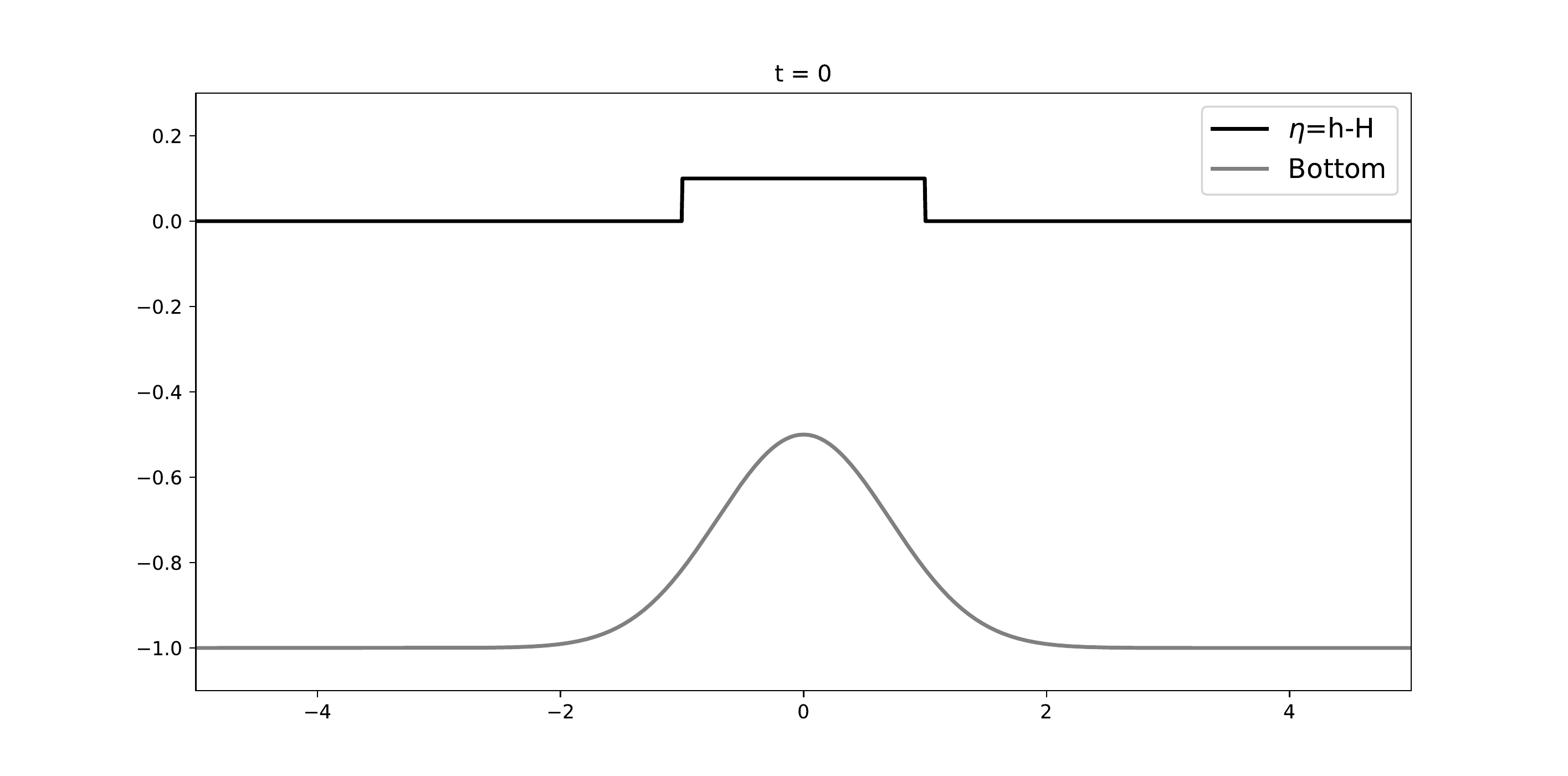}
     \caption{ Shallow water equations without friction: Test 3. Initial condition.}
     \label{sw_shock}
  \end{center}
 \end{figure}
 
A reference solution with a 1600-cell mesh using the SIWBM2 with piecewise linear reconstruction for $\widetilde Q_i$ has been obtained. Figures \ref{sw_shock_t02}-\ref{sw_shock_t05} show the numerical solutions obtained with the different methods at times $t=0.2, 0.5$.

\begin{figure}[ht]
 \begin{center}
   \includegraphics[width=0.8\textwidth]{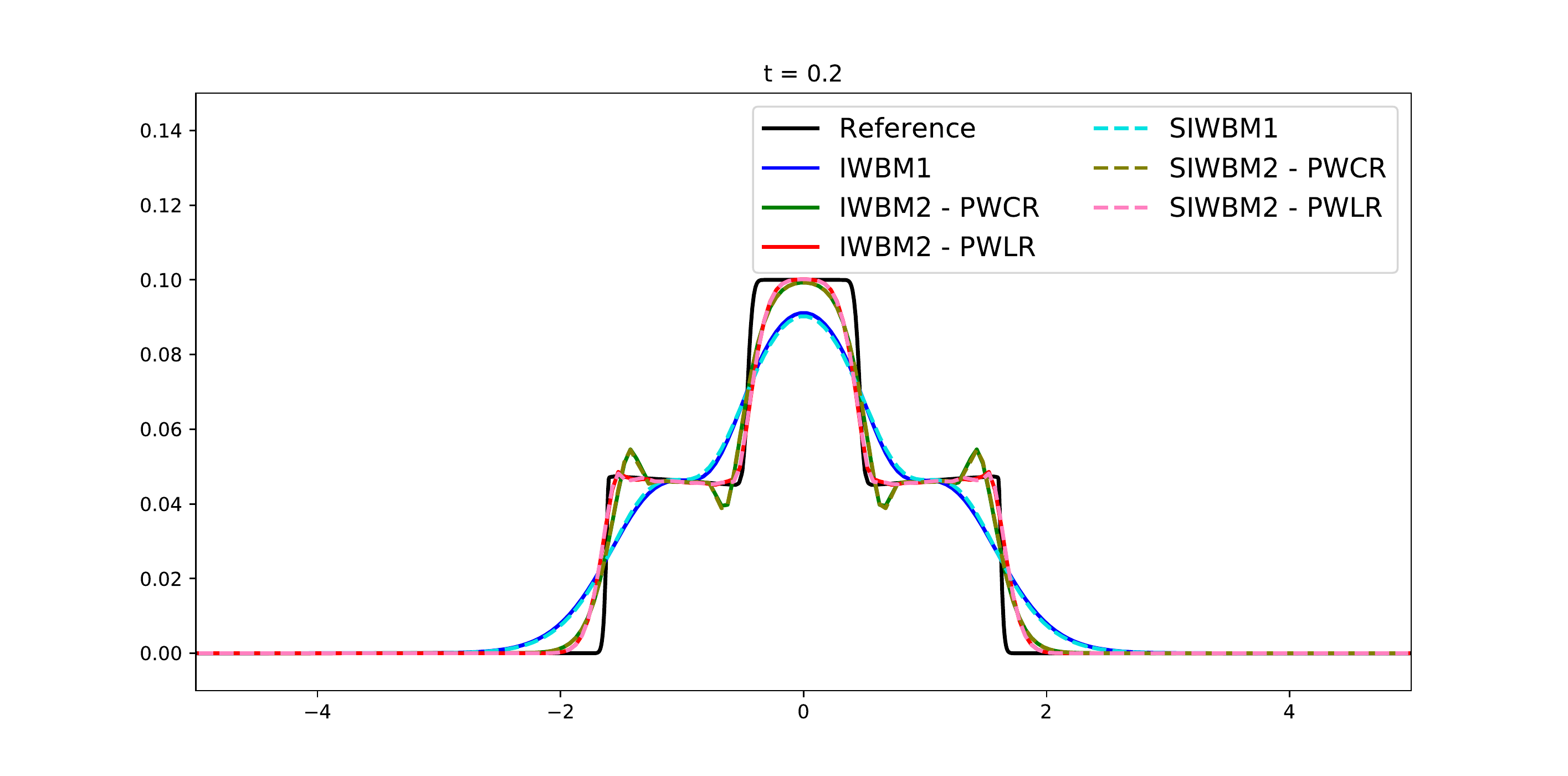}
   \includegraphics[width=0.8\textwidth]{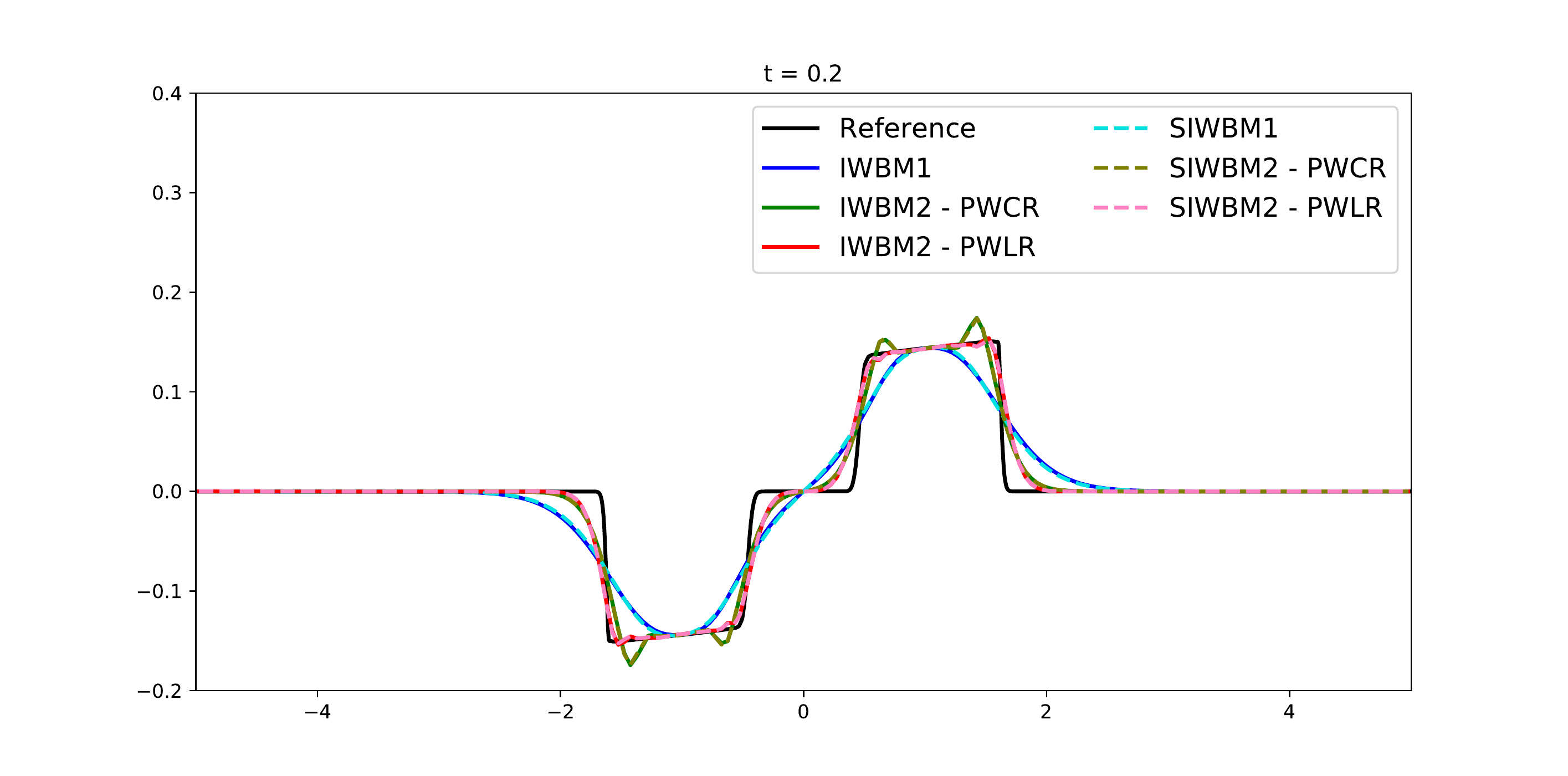}
     \caption{  Shallow water equations without friction: Test 3. Numerical solutions for IWBM$p$, SIWBM$p$, $p=1,2$ with CFL$=2$ at $t=0.2$. Top: $\eta$. Bottom: $q$.} 
     \label{sw_shock_t02}
  \end{center}
 \end{figure}

 \begin{figure}[ht]
 \begin{center}
   \includegraphics[width=0.8\textwidth]{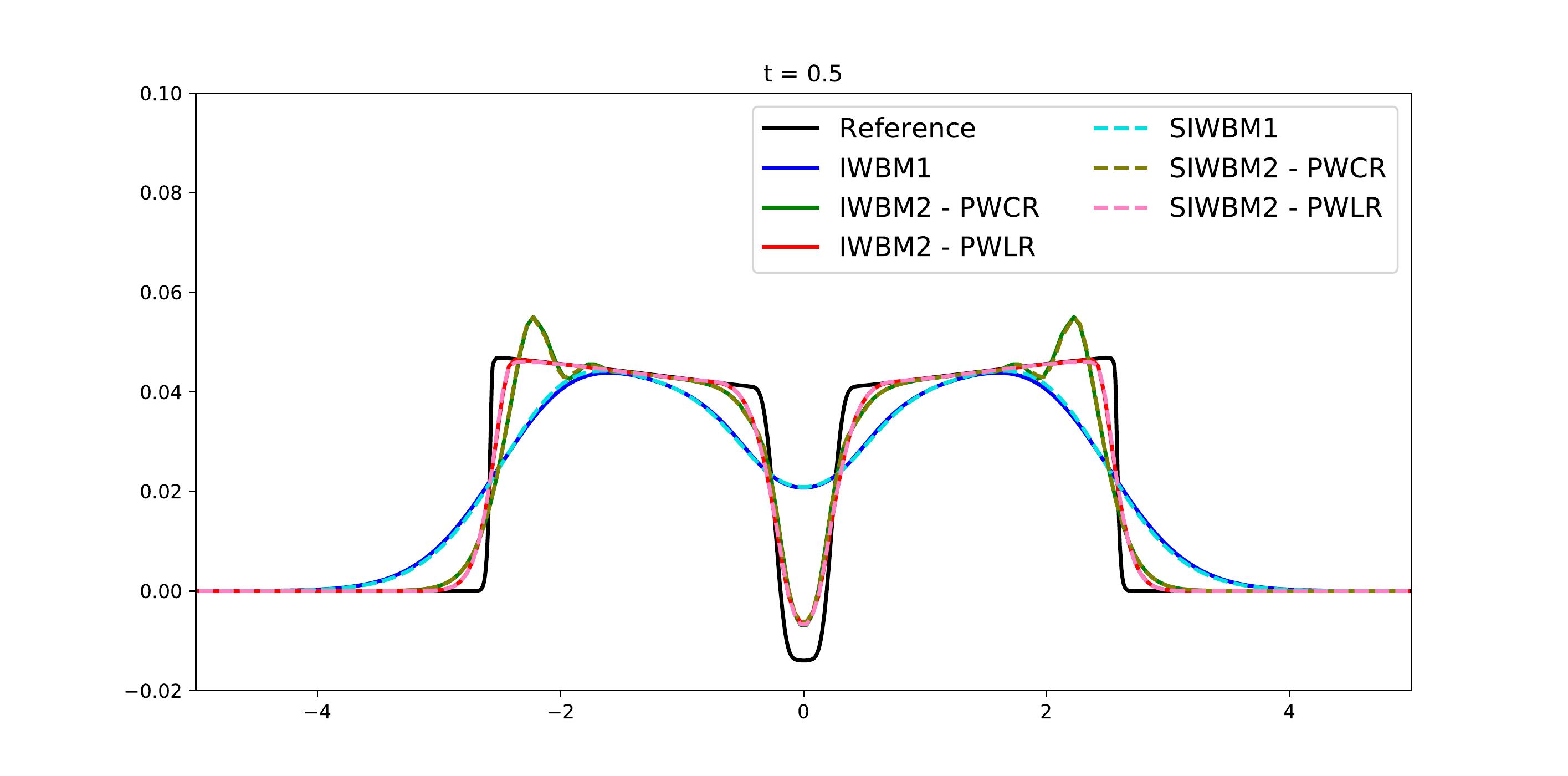}
   \includegraphics[width=0.8\textwidth]{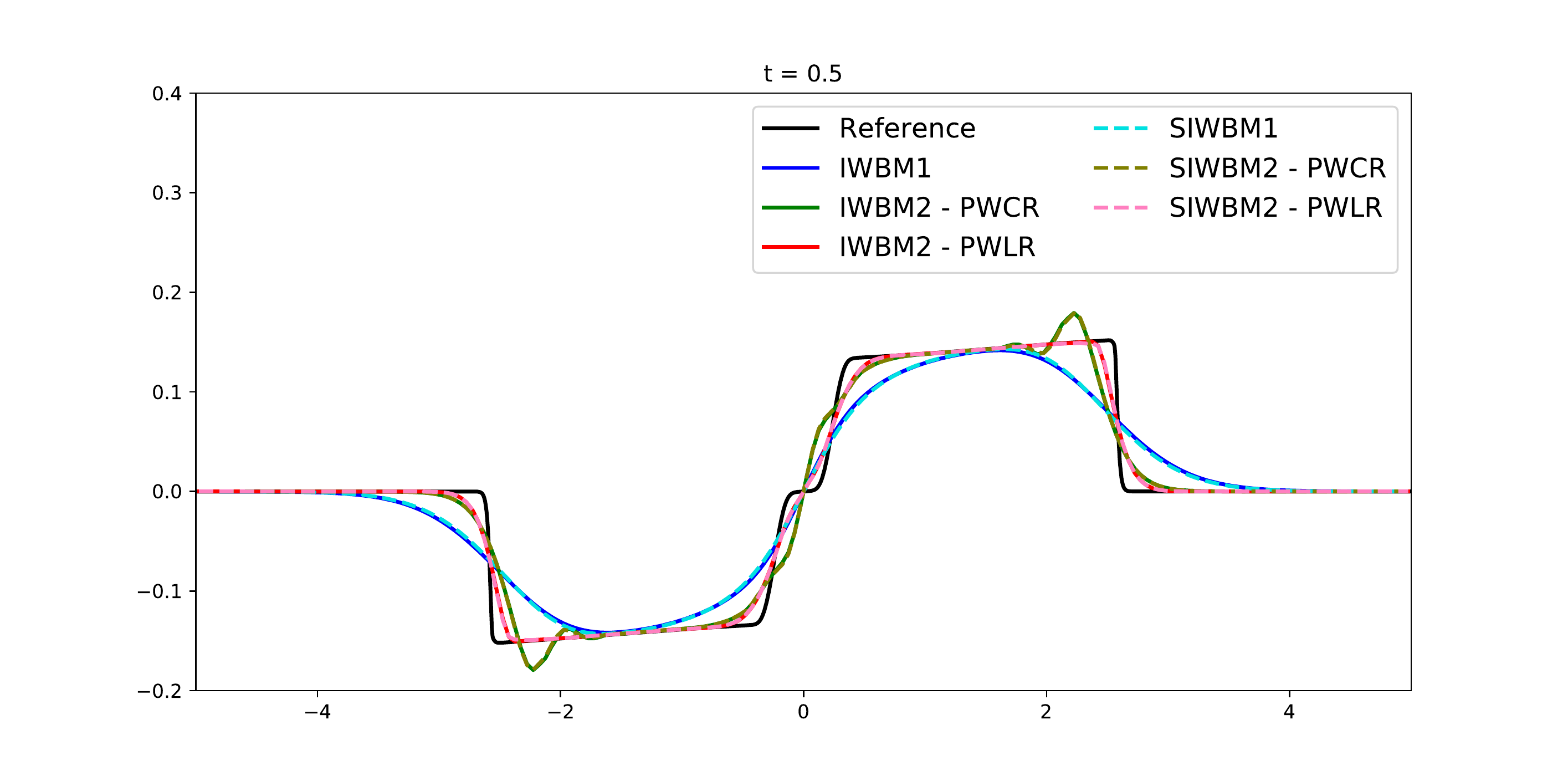}
     \caption{ Shallow water equations without friction: Test 3. Numerical solutions for IWBM$p$, SIWBM$p$, $p=1,2$ with CFL$=2$ at $t=0.5$. Top: $\eta$. Bottom: $q$.} 
     \label{sw_shock_t05}
  \end{center}
 \end{figure}

   Observe that when using a piecewise constant reconstruction $\widetilde{Q}_i(x)$, the numerical solution is less accurate and more oscillatory.

   
   \subsubsection{Test 4: convergence to a steady state}
   The goal now is to compare the ability of explicit and implicit schemes to reach a stationary solution as time increases.  We consider again the model without friction ($k=0$). Only the first-order methods EXWBM1 and IWBM1 are considered here.

   The space interval is $ [0,3]$ and the depth function is given again by \eqref{sw_fondo_subcrit}. The initial condition is $h(x,0)=2.0$ and $q(x,0)=0.0$. The  boundary conditions are the following
$$
q(0,t) = 1.0, \quad h(3,t) = 2.0.
$$
A 100-cell mesh is considered. The numerical solution is run until a stationary state is reached: the numerical method is stopped if 
\begin{equation}
    \displaystyle\frac{\max_i|U_i^{n+1}-U_i^n|}{\Delta t}<\varepsilon, 
\end{equation}
where $\varepsilon$ is a fixed threshold. In this test, we take $\varepsilon=$1e-12. Figure \ref{sw_converg} shows the stationary solution and Table  \ref{sw_converg_times} shows, for every numerical method, the time in seconds needed to reach a steady state, the difference in $L^1$-norm between the reached steady state and the subcritical stationary solution that  solves the problem
\begin{equation}\label{reached_ss}
\begin{cases}
q_x=0,\\
(\displaystyle - u^2 + gh) h_x= ghH_x,\\
h(3)=2.0, \, q(0)=1.0,\\
\end{cases}
\end{equation}
the CPU times, the total number of iterations in time associated to the time step $\Delta t$ and the total number of iterations of the fixed-point algorithm applied to solve the nonlinear problems are in the case of fully implicit schemes.
As expected, the implicit methods converge faster to the stationary solution. 
Notice that if the problem is just to capture the global stationary solution, then first order schemes may be perfectly adequate. However the interplay between time accuracy and the final stationary solution (in case the latter depends on the initial conditions as well and not only on the boundary conditions) is not obvious and it will be explored in future work.

\begin{figure}[ht]
 \begin{center}
   \includegraphics[width=0.65\textwidth]{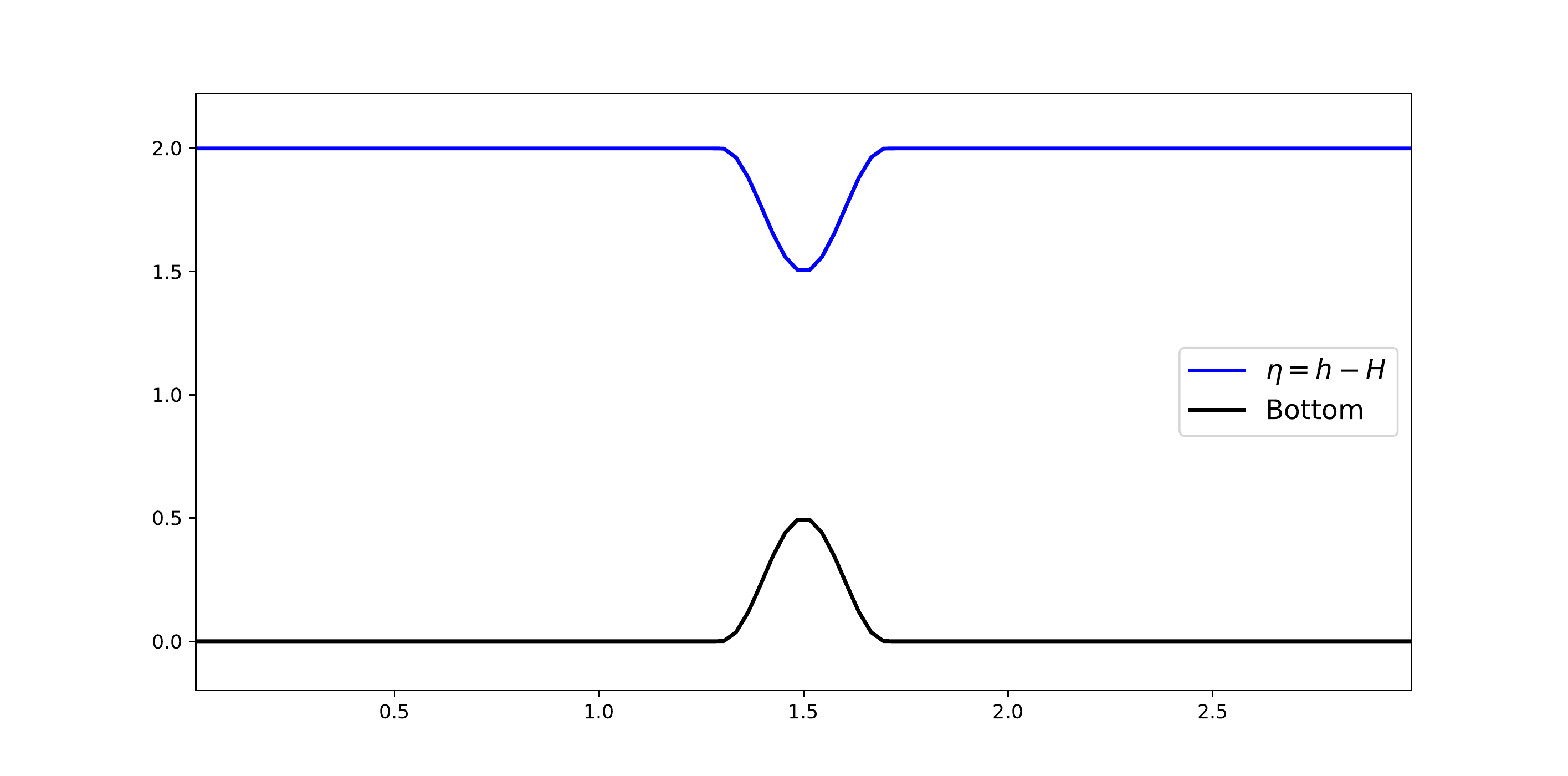}
    \includegraphics[width=0.65\textwidth]{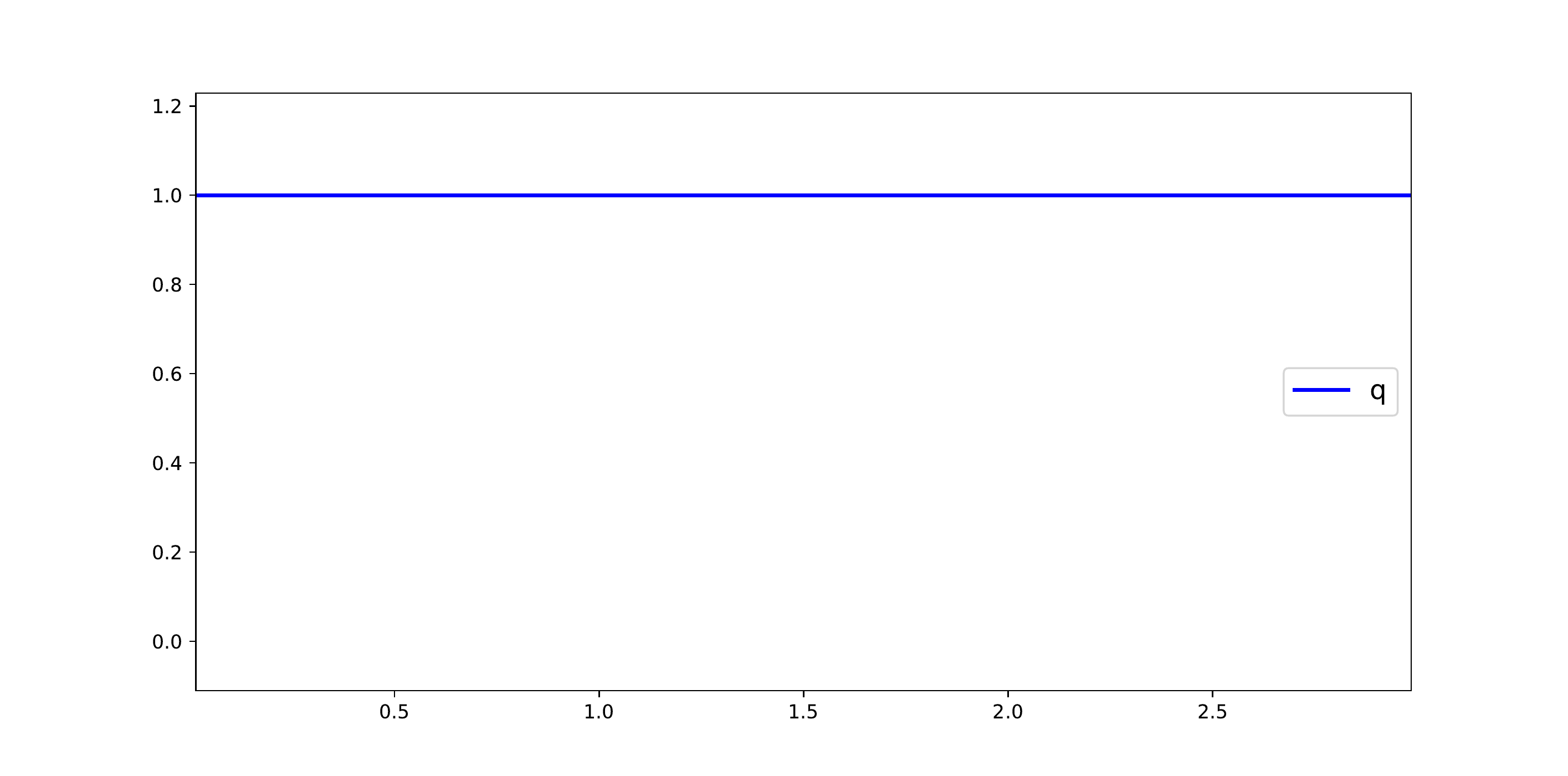}
     \caption{ Shallow water equations without friction: Test 4. Stationary solution.}
     \label{sw_converg}
  \end{center}
 \end{figure}

 \begin{table}[ht]
\centering
\begin{tabular}{|c|c|cc|c|c|c|} \hline
\multicolumn{7}{|c|}{\textbf{EXWBM1}} \\\hline
CFL & Times of  & \multicolumn{2}{|c|}{Errors of convergence}& CPU & Iterations & Fixed-point \\
& convergence & h & q & times & in time & iterations\\\hline

0.5 & 177.91 &2.55e-14 & 1.61e-13 & 80.90 & 52465 & - \\
0.99 & 159.34 &1.35e-13 & 1.29e-12 & 49.35 & 29586 & - \\
 \hline
\multicolumn{7}{|c|}{\textbf{IWBM1}} \\\hline
CFL & Times of  & \multicolumn{2}{|c|}{Errors of convergence}& CPU  & Iterations & Fixed-point \\
& convergence & h & q &times &in time & iterations \\\hline
2 & 129.51 & 1.96e-13 &1.65e-12 & 186.96 & 10660 & 97720\\
10 & 85.84 & 2.17e-13 & 1.83e-12 & 38.40& 1413& 36265 \\
20 &  64.04 &1.55e-13 &1.17e-12 &21.57& 527&22606\\
50 & 41.64 &4.42e-14 & 1.62e-12 & 20.60& 138 &26194\\\hline
\end{tabular}
 \caption{ Shallow water equations without friction: Test 4. Time needed to reach a steady state, differences in $L^1$-norm between the reached steady state and the subcritical stationary solution which solves the problem \eqref{reached_ss}, CPU times, total number of iterations in time associated to the time step $\Delta t$ and  total number of iterations of the fixed-point algorithm applied to solve the nonlinear problems are in the case of fully implicit schemes.} 
\label{sw_converg_times}
\end{table}

Notice that, for CFL=50, in spite of the fact that the total number of iterations of the fixed-point algorithm is bigger than in the case of CFL=20, less computational effort is required since the total number of well-balanced reconstructions is smaller.

\subsubsection{Test 5: stationary solution for the model with friction}
Let us check the well-balanced property of the methods for the model with friction. In this test, the Manning friction is $k=0.01$ and the space interval is $[0,1]$. The depth function is given by 
\begin{equation}\label{swf_fondo}
H(x)= 1 - \displaystyle \frac{1}{2} \frac{e^{\cos(4 \pi x)} -e^{-1}}{e - e^{-1}}.
\end{equation}
We consider a supercritical stationary  solution: the solution of  \eqref{u'=Gswf} with initial conditions $q(0)=3$ and $h(0)=0.3$ (see Figure \ref{swf_sssup_cini}), obtained numerically by solving system \eqref{u'=Gswf} using the mid-point rule (see \cite{gomez2021collocation}). We consider a 100-cell mesh  and the final time is $t=1$. Notice that, at variance with the low Froude number stationary solutions, in the supercritical case the profile of the free surface is almost parallel to the bottom profile, so that $h(x)$ is almost constant.
Table \ref{swF_wbcheck_errors} shows the $L^1$-errors between the initial and the approximated cell-averages at time $t = 1$ given by SIWBM1, IWBM1, SIWBM2 and IWBM2 with piecewise contant (PWCR) or piecewise linear (PWLR)  reconstruction $\widetilde Q_i$.
\begin{figure}[ht]
 \begin{center}
   \includegraphics[width=0.8\textwidth]{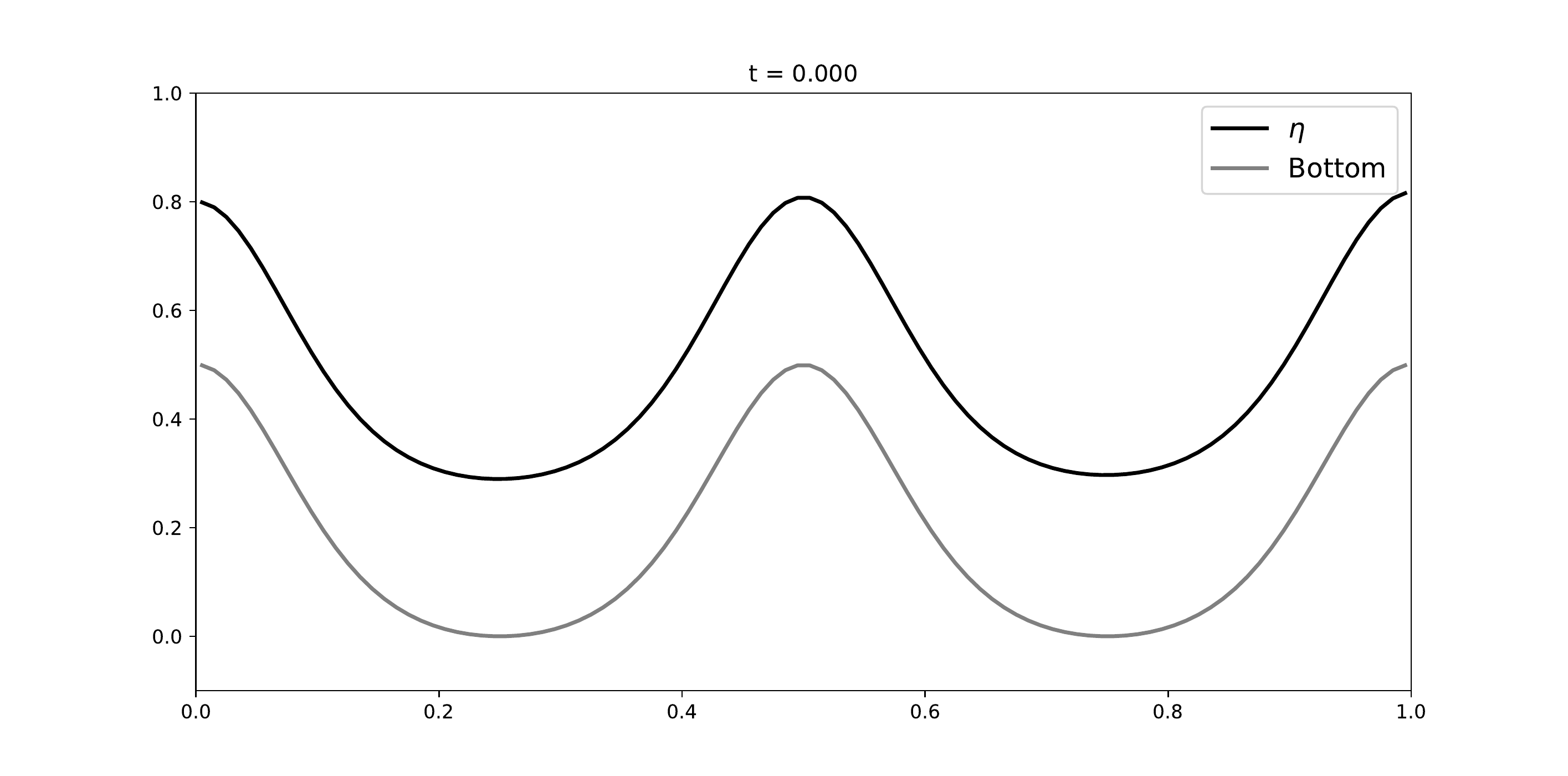}
     \caption{ Shallow water equations with friction: Test 5. Initial condition: supercritical stationary solution. Free surface and bottom.} 
     \label{swf_sssup_cini}
  \end{center}
 \end{figure}

 \begin{table}[ht]
\centering
\begin{tabular}{|cc|cccc|} \hline\multicolumn{6}{|c|}{\textbf{Implicit methods}}\\\hline
\multicolumn{2}{|c|}{IWBM1} & \multicolumn{4}{c|}{IWBM2}  \\
\multicolumn{2}{|c|} {  }&   \multicolumn{2}{c}{PWCR} & \multicolumn{2}{c|}{PWLR} \\ \hline
h& q & h & q &h & q \\\hline
6.11e-16 & 8.88e-16 & 9.44e-16  & 9.36e-15 & 6.66e-16 & 6.22e-15 \\\hline
 \multicolumn{6}{c}{}\\
 \multicolumn{6}{c}{}\\\hline
\multicolumn{6}{|c|}{\textbf{Semi-implicit methods}}\\\hline
\multicolumn{2}{|c|}{SIWBM1} & \multicolumn{4}{c|}{SIWBM2}  \\
\multicolumn{2}{|c|} {  }&   \multicolumn{2}{c}{PWCR} & \multicolumn{2}{c|}{PWLR} \\ \hline
h& q & h & q &h & q \\\hline
7.21e-16 & 6.66e-15 & 9.44e-16  & 9.76e-15 & 8.33e-16 & 6.21e-15 \\\hline
\end{tabular}
 \caption{ Shallow water equations with friction: Test 5. Differences in $L^1$-norm between the stationary and the numerical solution  at $t=1$ for IWBM1, SIWBM1, IWBM2 and SIWBM2 with piecewise contant (PWCR) or piecewise linear (PWL) reconstruction $\widetilde{Q}_i$ for a 100-cell mesh.
} 
\label{swF_wbcheck_errors}
\end{table}

\subsubsection{Test 6: perturbation of a stationary solution for the model with friction}
In this test, the Manning friction is again $k=0.01$. The depth function is given by \eqref{swf_fondo}. We consider a perturbation of the supercritical stationary  solution: the initial condition $U_0(x)=[h_0(x),q_0(x)]^T$ given by 
$$
h_0(x)=\begin{cases}
h^*(x)+0.05,
& \mbox{if $x \in \left[ \displaystyle \frac{2}{7} , \displaystyle \frac{3}{7} \right] \cup \left[ \displaystyle \frac{4}{7} , \displaystyle \frac{5}{7} \right]$,}\\
h^*(x), & \mbox{otherwise,}
\end{cases}
$$
$$
q_0(x)=\begin{cases}
q^*(x)+0.5,
& \mbox{if $x \in \left[ \displaystyle \frac{2}{7} , \displaystyle \frac{3}{7} \right] \cup \left[ \displaystyle \frac{4}{7} , \displaystyle \frac{5}{7} \right]$,}\\
q^*(x), & \mbox{otherwise,}
\end{cases}
$$ 
where $U^*(x)=[h^*(x),q^*(x)]^T$ is the stationary solution considered in Test 5 (see Figure \ref{swf_pertsup_cini}). We consider a 100-cell mesh  and the numerical simulation  is  run until  $t = 2$. A reference solution computed with a 800-cell mesh using SIWBM2 with piecewise constant reconstruction $\widetilde Q_i$ has been considered.  Figures \ref{swf_pertsuper_imex_001}-\ref{swf_pertsuper_imex_005} show the evolution of the perturbation at times $t=0.01$ and $t=0.05$ for SIWBM1 and SIWBM2 with piecewise contant (PWCR) or piecewise linear (PWLR) reconstruction $\widetilde{Q}_i$ and Table \ref{swF_wbcheck_errors_imex} shows the differences in $L^1$-norm between the stationary and the numerical solutions  at $t=2$. As expected, the stationary solution is recovered with machine precision.
\begin{figure}[ht]
 \begin{center}
   \includegraphics[width=0.65\textwidth]{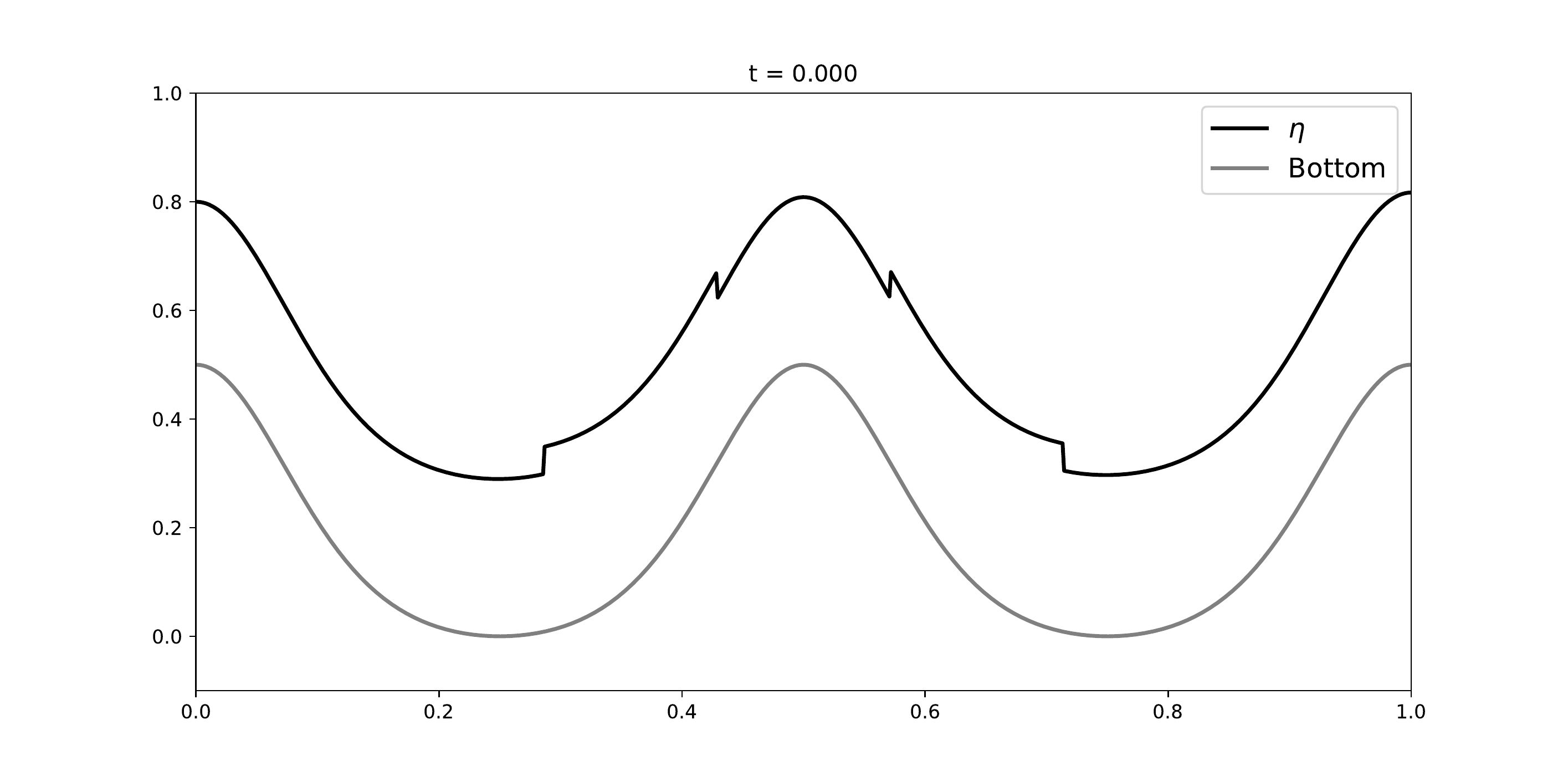}
   \includegraphics[width=0.65\textwidth]{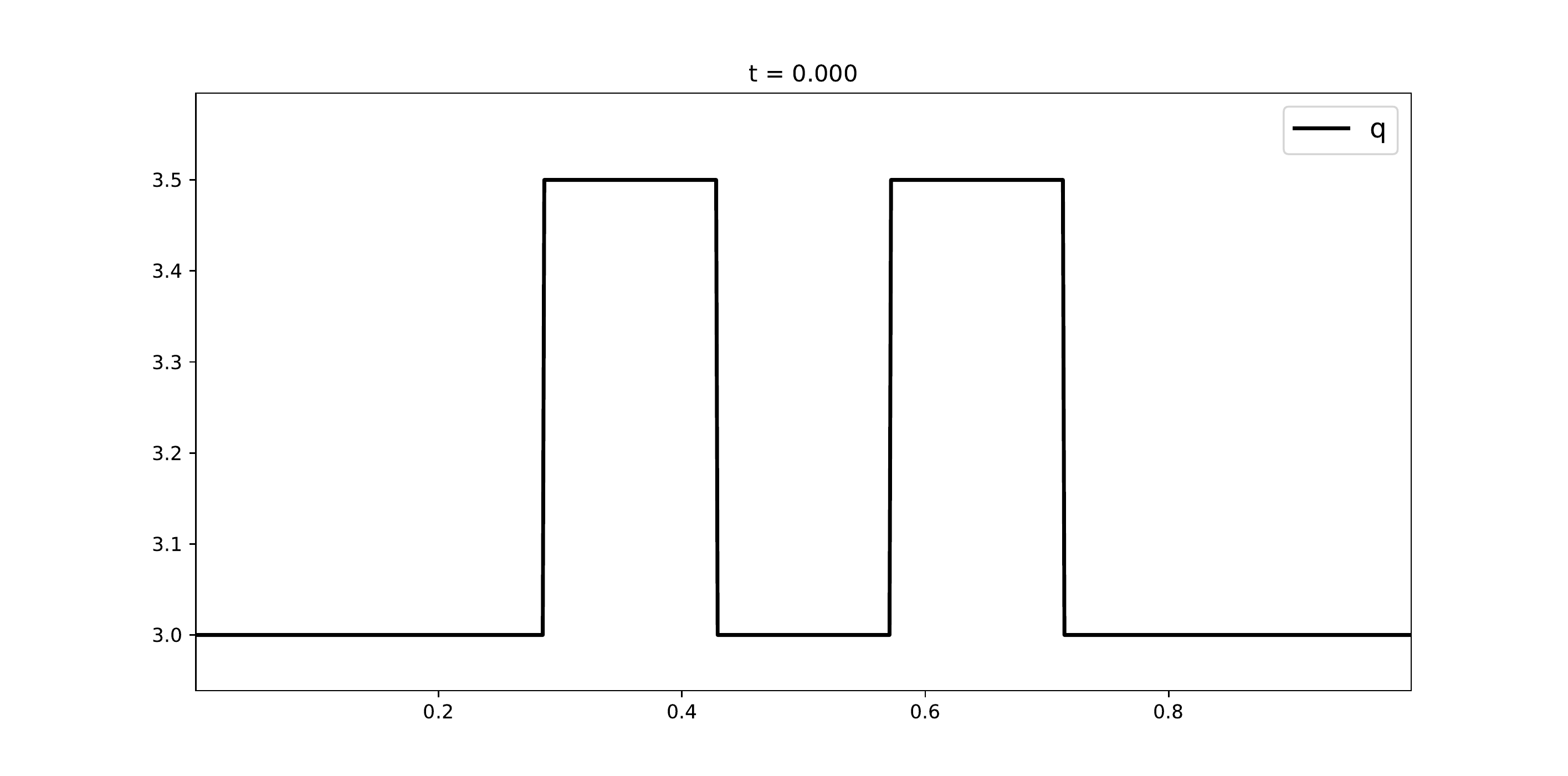}
     \caption{ Shallow water equations with friction: Test 6. Initial condition: perturbation of a supercritical stationary solution. Top: $\eta$. Bottom: $q$.} 
     \label{swf_pertsup_cini}
  \end{center}
 \end{figure}
 \begin{figure}[ht]
 \begin{center}
   \includegraphics[width=0.65\textwidth]{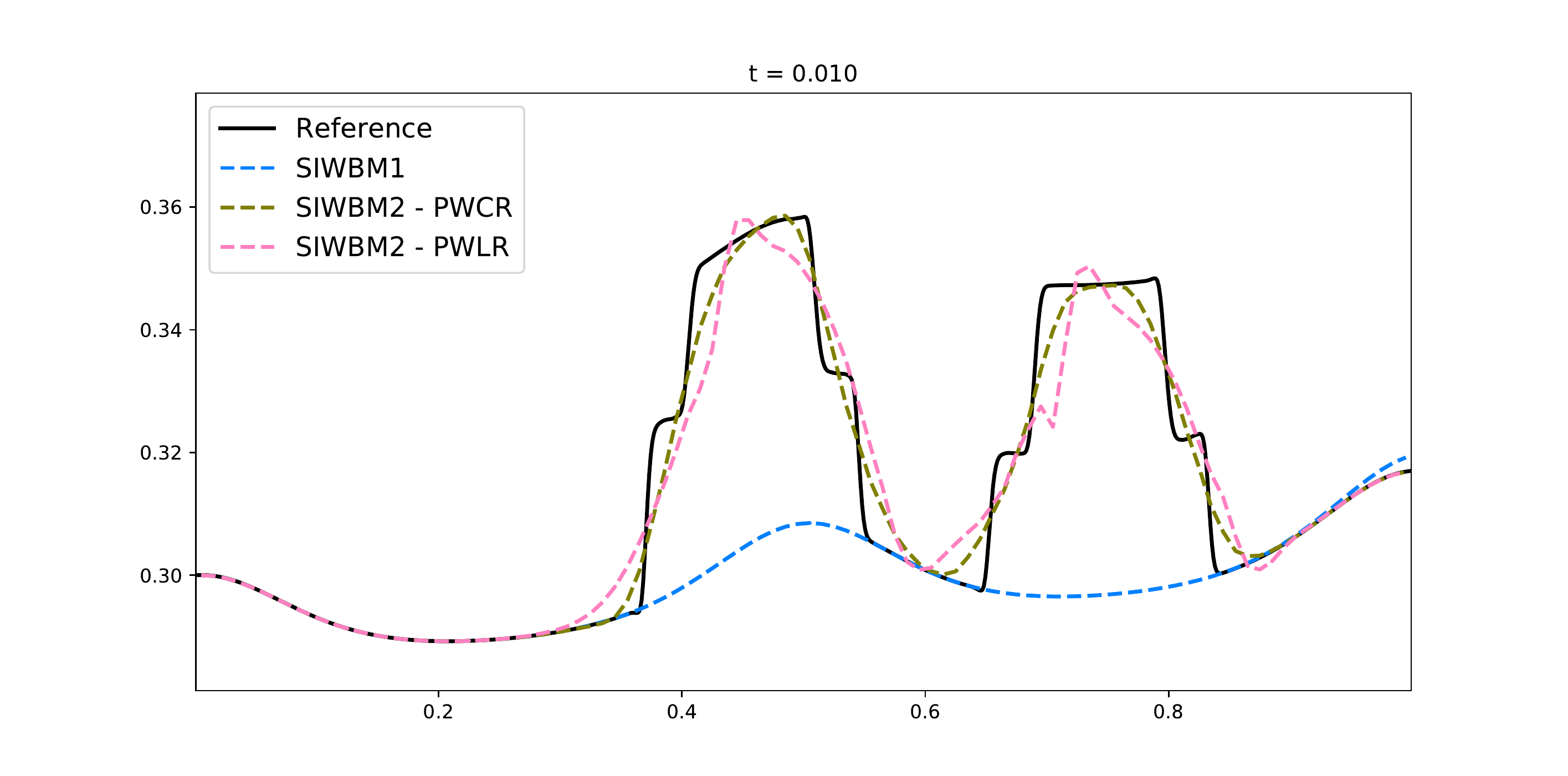}
   \includegraphics[width=0.65\textwidth]{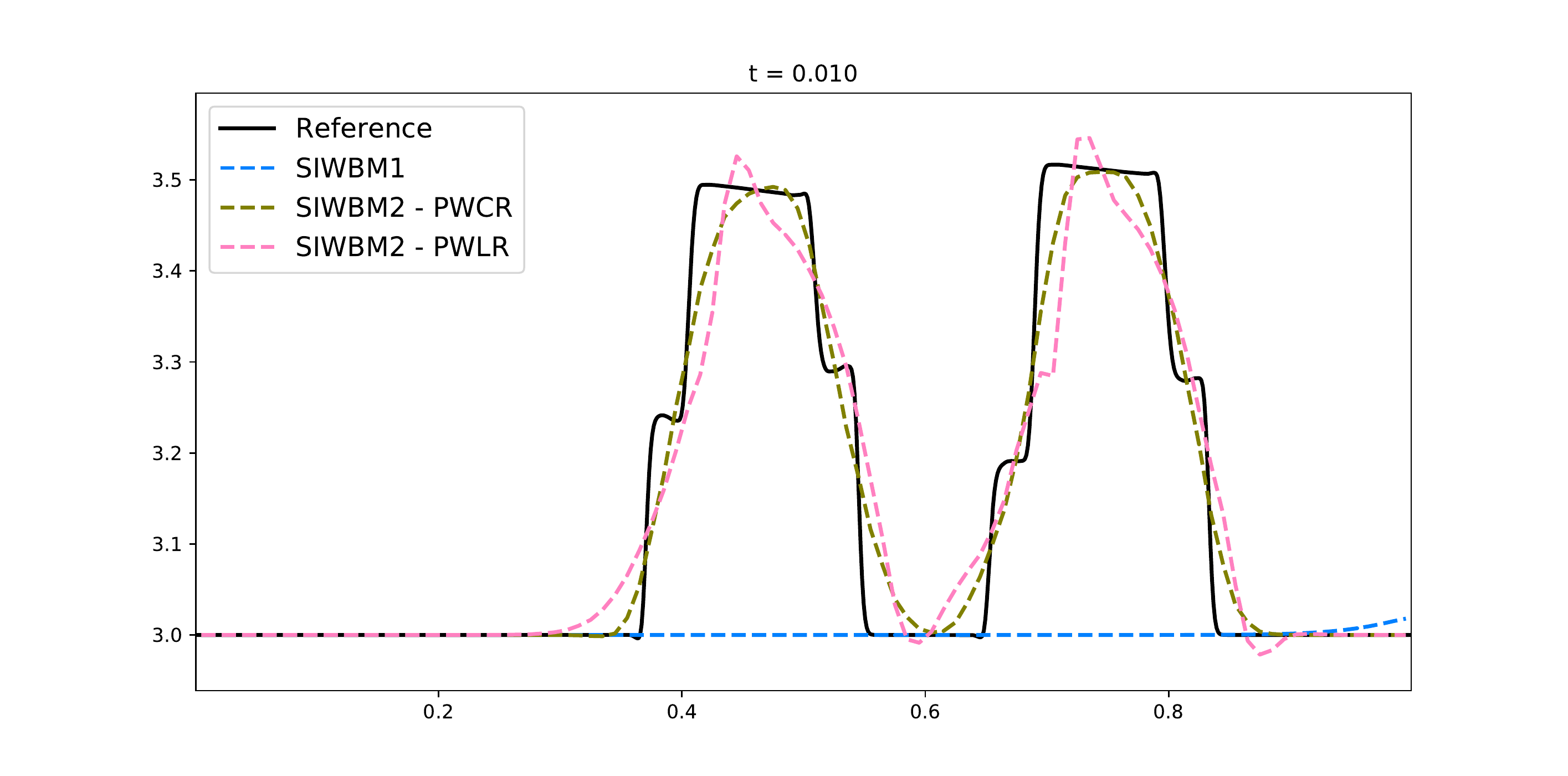}
     \caption{ Shallow water equations with friction: Test 6. Numerical solutions for SIWBM$p$, $p=1,2$ with CFL$=0.9$ at $t=0.01$. Top: $h$. Bottom: $q$.
     } 
     \label{swf_pertsuper_imex_001}
  \end{center}
 \end{figure}
 \begin{figure}[ht]
 \begin{center}
   \includegraphics[width=0.65\textwidth]{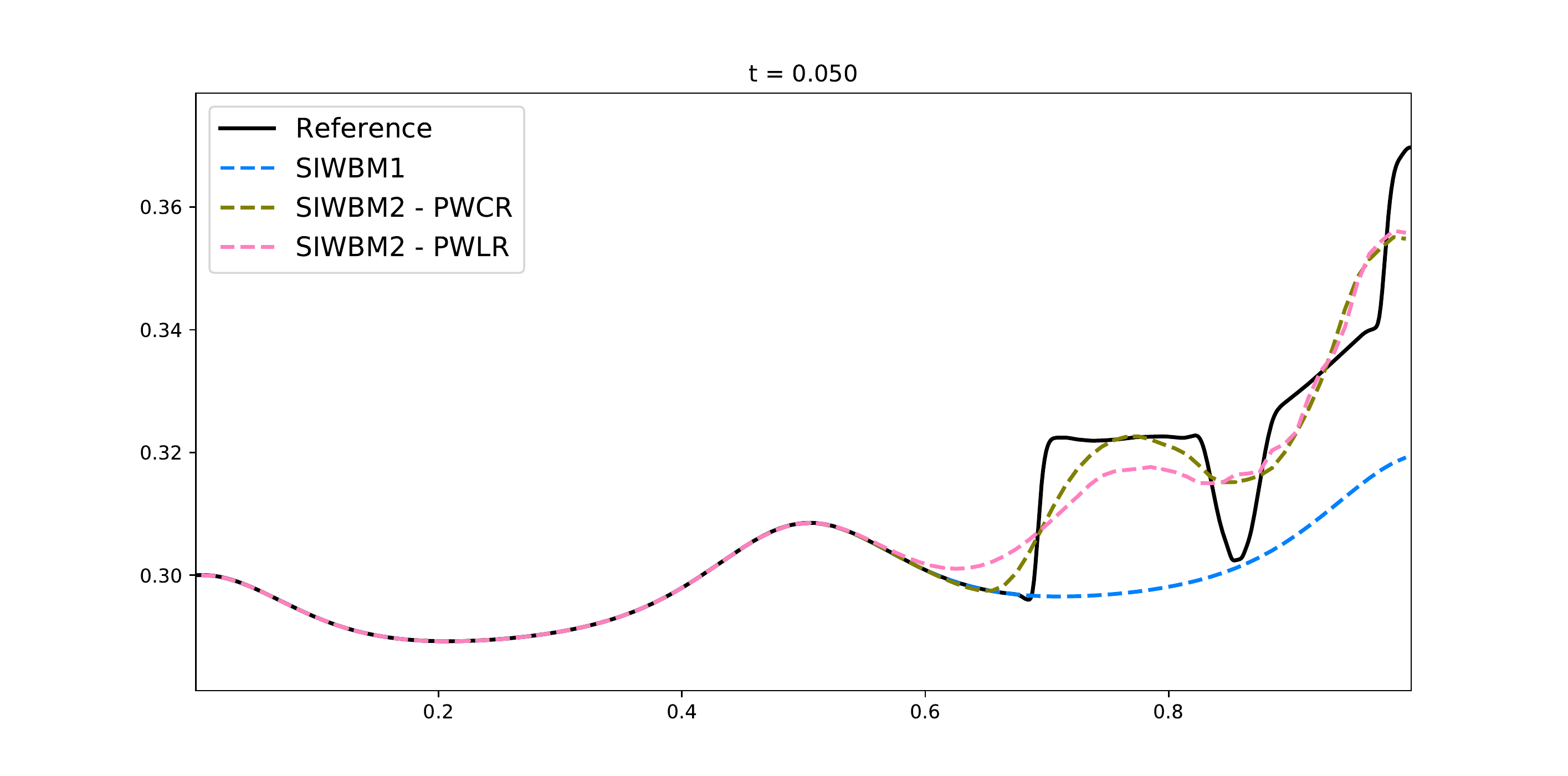}
   \includegraphics[width=0.65\textwidth]{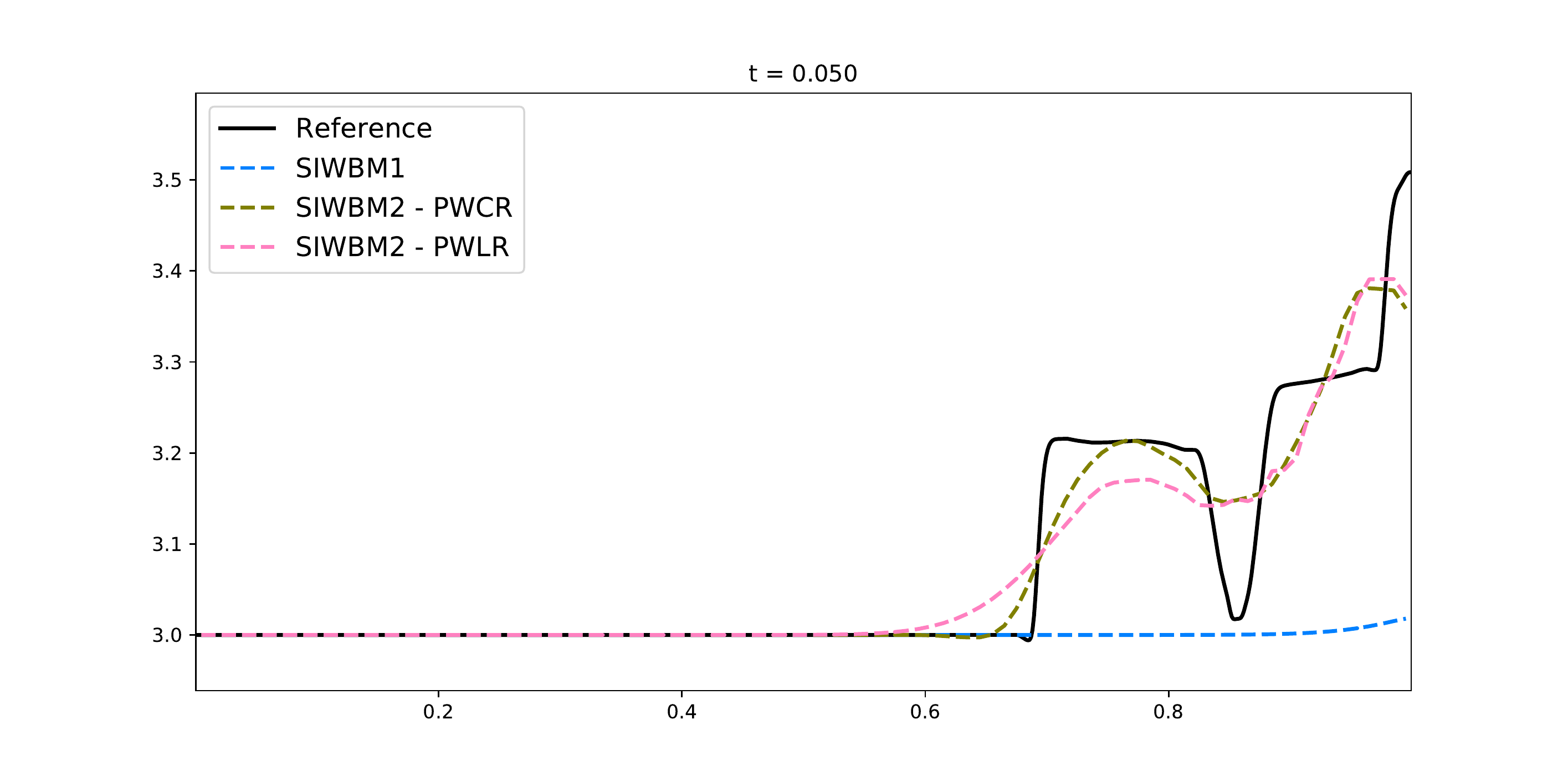}
     \caption{ Shallow water equations with friction: Test 6. Numerical solutions for SIWBM$p$, $p=1,2$ with CFL$=0.9$ at $t=0.05$. Top: $h$. Bottom: $q$.} 
     \label{swf_pertsuper_imex_005}
  \end{center}
 \end{figure}
  \begin{table}[ht]
\centering
\begin{tabular}{|cc|cccc|} 
\hline
\multicolumn{6}{|c|}{\textbf{Semi-implicit methods}}\\\hline
\multicolumn{2}{|c|}{SIWBM1} & \multicolumn{4}{c|}{SIWBM2}  \\
\multicolumn{2}{|c|} {  }&   \multicolumn{2}{c}{PWCR} & \multicolumn{2}{c|}{PWLR} \\ \hline
h& q & h & q &h & q \\\hline
9.99e-16 & 1.15e-14 & 4.44e-16  & 1.33e-15 & 6.10e-16 & 5.32e-15 \\\hline
\end{tabular}
 \caption{ Shallow water equations with friction: Test 6. Differences in $L^1$-norm between the stationary and the numerical solution  at $t=2$ for SIWBM1 and SIWBM2 with piecewise contant (PWCR) or piecewise linear (PWLR) reconstruction $\widetilde{Q}_i$ for a 100-cell mesh.} 
\label{swF_wbcheck_errors_imex}
\end{table}

Fully implicit schemes have been also considered with CFL=2, recovering with machine precision the supercritical stationary solution at the final time $t=2s$ (see Table \ref{swF_wbcheck_errors_imp}, where the errors $L^1$-norm between the stationary and the numerical solutions  at $t=2$ for IWBM1 and IWBM2 with piecewise contant (PWCR) or piecewise linear (PWLR) reconstruction $\widetilde{Q}_i$  are shown).

  \begin{table}[ht]
\centering
\begin{tabular}{|cc|cccc|} 
\hline
\multicolumn{6}{|c|}{\textbf{Implicit methods}}\\\hline
\multicolumn{2}{|c|}{IWBM1} & \multicolumn{4}{c|}{IWBM2}  \\
\multicolumn{2}{|c|} {  }&   \multicolumn{2}{c}{PWCR} & \multicolumn{2}{c|}{PWLR} \\ \hline
h& q & h & q &h & q \\\hline
5.00e-16 & 4.41e-16 & 1.50e-15  & 1.51e-14 & 8.33e-16 & 7.55e-15 \\\hline
\end{tabular}
 \caption{ Shallow water equations with friction: Test 6. Differences in $L^1$-norm between the stationary and the numerical solution  at $t=2$ for IWBM1 and IWBM2 with piecewise contant (PWCR) or piecewise linear (PWLR) reconstruction $\widetilde{Q}_i$ for a 100-cell mesh.} 
\label{swF_wbcheck_errors_imp}
\end{table}

\clearpage
\section{Conclusions and forthcoming work}
Following some previous work of the authors \cite{sinum2008, CastroPares2019, GomezCastroPares2020, gomez2021well, gomez2021collocation}, we have developed a general procedure to design high-order implicit and semi-implicit numerical schemes for any one-dimensional system of  balance laws.  Note that the main ingredient of these methods is a well-balanced reconstruction operator. A general result proving the well-balanced property of these numerical methods is  stated. We checked the new formulation with several numerical tests, considering scalar problems such as the linear transport equation and the Burgers equation, and more complex systems such as shallow water in presence of variable bathymetry and  Manning friction. 
Notice that, when both the flux and the source term of \eqref{PDE_generalproblem} are (equally) stiff, the system may relax to a stationary solution of the ODE system \eqref{ODE_stationarysolutions}
in a very short time. If one is interested in efficiently capturing the stationary solution, then it is advisable to adopt an implicit (or semi-implicit) scheme which is at the same time well-balanced. This is shown in a numerical test for the shallow water model.

Future work will include applications to more general systems whose source contains a stiff relaxation and a non-stiff term, i.e.\ systems of the form 
\begin{equation}
u_t + f(u)_x = \frac{1}{\epsilon}S(u) + G(u,x),
\label{balanceG}
\end{equation}
that in the limit of vanishing $\epsilon$ relaxes to a lower dimensional system of balance laws of the form 
\begin{equation}
v_t + \bm{f}(v)_x = \bm{g}(v,x),
\label{balanceg}
\end{equation}
where where $v(x,t) = Qu(x,t)\in \mathbb{R}^n$, $n<N$, $Q\in\mathbb{R}^{n\times N}$, $QS(u) = 0$, and  $\bm{f} = Qf(E(v))$, with $u = E(v)$, $\bm{g}(u,x) = QG(E(u),x)$. In such cases the limit equation admits non-trivial stationary solutions which must be accurately approximated. Our aim will be to design numerical schemes for systems \eqref{balanceG} which become consistent and  well-balanced schemes for systems \eqref{balanceg} as the relaxation parameter vanishes, which are said to be {\em Asymptotic Preserving and Well-Balanced (APWB)\/} (see \cite{jin1995relaxation, jin2010asymptotic, chen1994hyperbolic,pareschi2005implicit}). A natural framework to define such numerical schemes is the combination of well-balanced finite-volume schemes and IMEX methods. 


A second important extension concerns the application of this framework to problems in more space dimensions, witht he specific goal to capture non trivial stationary solutions of systems of balance laws, along the lines of the pioneering work of Moretti and Abbett \cite{moretti1966time}, who captured the stationary flow around a blunt body by looking for the stationary solutions of a time dependent problem.

\section*{Acknoledgements}{
We would like to thank  M. L\'opez-Fern\'andez for the linear stability analysis performed in Test 2 for the transport equation.
This work is partially supported by projects RTI2018-096064-B-C21 funded by  MCIN/AEI/10.13039/501100011033 and  ``ERDF A way of making Europe'',  projects P18-RT-3163 of Junta de Andaluc\'ia  and UMA18-FEDERJA-161 of Junta de Andaluc\'ia-FEDER-University of M\'alaga. 
 G. Russo and S.Boscarino acknowledge partial support from  the Italian Ministry of University and Research (MIUR), PRIN Project 2017 (No. 2017KKJP4X entitled “Innovative numerical methods for evolutionary partial differential equations and applications”. 
 I. Gómez-Bueno is also supported by a Grant from “El Ministerio de Ciencia, Innovación y Universidades”, Spain (FPU2019/01541) funded by 
 MCIN/AEI/10.13039/501100011033 and ``ESF Investing in your future''.
}

\bibliographystyle{plain} 
\bibliography{references}


\end{document}